\documentclass[final,leqno]{siamltex}
\pdfoutput=1
\usepackage{amssymb,amsmath,bbm}
\usepackage[usenames,dvipsnames]{color}
\usepackage{graphicx,dblfloatfix}
\usepackage{mathrsfs}
\usepackage{subfigure}
\usepackage{hyperref}

\newtheorem{remark}[theorem]{Remark}
\newtheorem{problemstatement}[theorem]{Problem}
\newtheorem{algorithmdefinition}[theorem]{Algorithm}

\newcommand{\reals}{\mathbb{R}} 
\newcommand{\realsnonnegative}{\mathbb{R}_{\ge 0}} 
\newcommand{\naturals}{\mathbb{N}}

\newcommand{\co}{\operatorname{co}}
\newcommand{\interior}{\operatorname{int}}
\newcommand{\dom}{\operatorname{dom}}

\newcommand{\primalsol}{\mathcal{X}}
\newcommand{\ctriggerset}{\mathcal{T}^c} 
\newcommand{\dtriggerset}{\mathcal{T}_i}
\newcommand{\ctriggersete}{\mathcal{T}^{c,e}} 
\newcommand{\ctriggersetsigma}{\mathcal{T}^{c,\sigma}} 
\newcommand{\ctriggersetzero}{\mathcal{T}^{c,0}} 
\newcommand{\dtriggersete}{\mathcal{T}^{e}_i}
\newcommand{\dtriggersetzero}{\mathcal{T}^{0}_i}
\newcommand{\dtriggersetrequest}{\mathcal{T}_i^{\operatorname{request}}}
\newcommand{\dtriggersetsend}{\mathcal{T}_i^{\operatorname{send}}}
\newcommand{\dtriggersetsynch}{\mathcal{T}^{\operatorname{synch}}_i}

\newcommand{\nomflow}{f}

\newcommand{\NormInf}[1]{\| #1 \|_{\infty}} 
\newcommand{\NormTwo}[1]{\| #1 \|_{2}} 
\newcommand{\until}[1]{\{ 1,\dots, #1\}} 
 
\newcommand{\neigh}{\mathcal{N}} 
\newcommand{\Kbound}{K_*} 
\newcommand{\oprocendsymbol}{\hbox{$\bullet$}}
\newcommand{\oprocend}{\relax\ifmmode\else\unskip\hfill\fi\oprocendsymbol}

\newcommand{\setdef}[2]{\{#1 \; |\; #2\}}
\newcommand{\map}[3]{#1: #2 \rightarrow #3}

\newcommand{\longthmtitle}[1]{\mbox{}\textup{\bf {(#1).}}}

\newcommand{\myclearpage}{\clearpage}
 \renewcommand{\myclearpage}{}

\begin{document}

\title{Distributed linear programming
  \\
  with event-triggered communication}



\author{Dean Richert \qquad Jorge Cort\'es \thanks{The authors are
    with the Department of Mechanical and Aerospace Engineering,
    University of California, San Diego, 9500 Gilman Dr, La Jolla CA
    92093, USA, {\tt\small \{drichert,cortes\}@ucsd.edu}} A
  preliminary version of this paper was submitted to the
  $53^{\rm{rd}}$ IEEE Conference on Decision and Control.}

\maketitle

\begin{abstract}
  We consider a network of agents whose objective is for the aggregate
  of their states to converge to a solution of a linear program in
  standard form.  Each agent has limited information about the problem
  data and can communicate with other agents at discrete time instants
  of their choosing. Our main contribution is the synthesis of a
  distributed dynamics and a set of state-based rules, termed
  triggers, that individual agents use to determine when to
  opportunistically broadcast their state to neighboring agents to
  ensure asymptotic convergence to a solution of the linear program.
  Our technical approach to the algorithm design and analysis
  overcomes a number of challenges, including establishing convergence
  in the absence of a common smooth Lyapunov function, ensuring that
  the triggers are detectable by agents using only local information,
  accounting for asynchronism in the state broadcasts, and ruling out
  various causes of arbitrarily fast state broadcasting. Various
  simulations illustrate our results.
\end{abstract}



\section{Introduction}\label{sec:introduction}

The global objective of many multi-agent systems can be formulated as
an optimization problem where the individual agents' states are the
decision variables. Due to the inherent networked structure of these
problems, much research has been devoted to developing local dynamics
for each agent that guarantee that the aggregate of their states
converge to a solution of the optimization problem.  From an analysis
viewpoint, the availability of powerful concepts and tools from
stability analysis makes continuous-time coordination algorithms
appealing. However, their implementation requires the continuous flow
of information among agents. On the other hand, discrete-time
algorithms are amenable to real-time implementation, but the selection
of the stepsizes to guarantee convergence has to take into account
worst-case situations, leading to an inefficient use of the network
resources.  In this paper, we seek to combine the advantages of both
approaches by designing a distributed algorithmic solution to linear
programming in standard form that combines continuous-time computation by
individual agents with opportunistic event-triggered communication
among neighbors.  Our focus on linear programming is motivated by its
importance in mathematical optimization and its pervasiveness in
multi-agent scenarios, with applications to task assignment, network
flow, optimal control, and energy storage, among others.

\emph{Literature review.} The present work has connections with three
main areas: distributed optimization, event-triggered control, and
switched and hybrid systems.  Distributed convex optimization problems
have many applications to networked systems, see
e.g.,~\cite{DPB-JNT:97,MGR-RDN:05,PW-MDL:09}, and this has motivated
the development of a growing body of work that includes
dual-decomposition~\cite{SS-SB-DG:07,LX-MJ-SPB:04}, the alternating
direction method of multipliers~\cite{EW-AO:12}, subgradient
projection algorithms~\cite{BJ-TK-MJ-KHJ:08,AN-AO:09,MZ-SM:13b},
auction algorithms~\cite{DPB-DAC:91}, and saddle-point
dynamics~\cite{DF-FP:10,BG-JC:14-tac}.  The
works~\cite{MB-GN-FB-FA:12,DR-JC:13-tac} propose algorithms
specifically designed for distributed linear
programming. In~\cite{MB-GN-FB-FA:12}, the goal is for agents to agree
on the global solution.  In~\cite{DR-JC:13-tac}, instead, the goal is
for the aggregate of agents' states to converge to a solution. All the
algorithms mentioned above are implemented in either continuous or
discrete time, the latter with time-dependent stepsizes that are
independent of the network state.  Instead, event-triggered control
seeks to opportunistically adapt the execution to the network state by
trading computation and decision-making for less communication,
sensing or actuation effort while guaranteeing a desired level of
performance, see e.g.,~\cite{WPMHH-KHJ-PT:12,XW-MDL:11,MMJ-PT:10}. In
this approach to real-time implementation, a key design objective,
besides asymptotic convergence, is to ensure the lack of an infinite
number of updates in any finite time interval of the resulting
event-triggered strategy. A few
works~\cite{PW-MDL:09,SSK-JC-SM:14-auto} have explored the design of
distributed event-triggered optimization algorithms for multi-agent
systems.  A major difference between event-triggered stabilization and
optimization is that in the former the equilibrium is known a priori,
whereas in the latter the determination of the equilibrium point is
the objective itself.  Finally, our work is related to the literature
on switched and hybrid systems~\cite{DL:03,JPH:04,RG-RGS-ART:12} where
discrete and continuous dynamical components coexist.  To the authors'
knowledge, this work is the first to consider event-triggered
implementations of state-dependent switched dynamical systems. A
unique challenge that must be overcome in this scenario is the fact
that the use of outdated state information may cause the system to
miss a mode switch, and this in turn may affect the overall stability
and performance.

\emph{Statement of contributions.}  The main contribution of the paper
is the design of a provably correct distributed dynamics which,
together with a set of distributed criteria to trigger state
broadcasts among neighbors, enable a group of agents to collectively
solve linear programs in standard form.  Our starting point is the
introduction of a novel distributed continuous-time dynamics for
linear programming based on an exact quadratic regularization and the
characterization of its solutions as saddle points of an augmented
Lagrangian function.  This distributed dynamics is discontinuous in
the agents' state because of the inequality constraints in the
original linear program.  Our approach to synthesize strategies that
rely only on discrete-time communication proceeds by having agents
implement the distributed continuous-time dynamics using a
sample-and-hold value of their neighbors' state. The key challenge is
then to identify suitable criteria to opportunistically determine when
agents should share information with their neighbors in order to
guarantee asymptotic convergence and persistency of the executions.
Because of the technical complexity involved in solving this
challenge, we structure our discussion in two steps, dealing with the
design first of centralized criteria and then of distributed ones.

Under our centralized event-triggered communication scheme, agents use
global knowledge of the network to determine when to synchronously
broadcast their state.  The characterization of the convergence
properties of the centralized implementation is challenging because
the original continuous-time dynamics is discontinuous in the agents'
state and the fact that its final convergence value (being the
solution of the optimization problem) is not known a priori, which
further complicates the identification of a common smooth Lyapunov
function.  Nevertheless, using concepts and tools from switched and
hybrid systems, we are able to overcome these obstacles by introducing
a discontinuous Lyapunov function and examining its evolution during
time intervals where state broadcasts do not occur

We build on our centralized design to synthesize a distributed
event-triggered communication law under which agents use local
information to determine when to broadcast their individual state.
Our strategy to accomplish this is to investigate to what extent the
centralized triggers can be implemented in a distributed way and
modify them when necessary.  In doing so, we face the additional
difficulty posed by the fact that the mode switches associated to the
discontinuity of the original dynamics are not locally detectable by
individual agents.  To address this challenge, we bound the evolution
of the Lyapunov function under mode mismatch and, based on this
understanding, design the distributed triggers so that any potential
increase of the Lyapunov function due to the mismatch is compensated
by the decrease in its value before the mismatch took place.
Moreover, the distributed character of the agent triggers leads to
asynchronous state broadcasts, which poses an additional challenge for
both design and analysis.  Our main result establishes the asymptotic
convergence of the distributed implementation and identifies
sufficient conditions for executions to be persistently flowing (that
is, state broadcasts are separated by a uniform time infinitely
often). We show that the asynchronous state broadcasts cannot be the
cause of non-persistently flowing executions and we conjecture that
all executions are in fact persistently flowing. As a byproduct of using a
hybrid systems modeling framework in our technical approach, we are
also able to guarantee that the global asymptotic stability of the
proposed distributed algorithm is robust to small
perturbations. Finally, simulations illustrate our results.

\emph{Organization.} Section~\ref{sec:prelim} introduces preliminary
notions. Section~\ref{sec:problem-statement} presents the problem
statement and network model. Section~\ref{sec:discontinuous-dynamics}
introduces the continuous-time dynamics on which we base our
event-triggered design.  Sections~\ref{sec:centralized}
and~\ref{sec:distributed} present, respectively, centralized and
distributed event-triggered mechanisms for communication and their
convergence analysis. Section~\ref{sec:sims} provides simulation
results in a multi-agent assignment problem and
Section~\ref{sec:conclusions} gathers our conclusions and ideas for
future work.

\section{Preliminaries} \label{sec:prelim}

This section introduces the notation and some notions from hybrid
systems and optimization employed throughout the paper.

\subsection{Notation}

We let $\reals$ and $\naturals$ denote the set of real and nonnegative
integer numbers, respectively. For a vector $x \in \reals^d$, $x \ge
0$ means that every component of $x$ is non-negative.  For $x \in
\reals^d$, $\NormTwo{x}$ and $\NormInf{x}$ denote its Euclidean and
$\infty$-norm, respectively. For a matrix $A \in \reals^{d_1 \times
  d_2}$, its $i^{\rm th}$ row is $a_{i}$, its $(i,j)$-element is
$a_{i,j}$, and its spectral radius is $\rho(A)$. 
Given sets $S_1,S_2 \subseteq \reals^d$, we let $S_1\setminus S_2$
denote the elements that are in $S_1$ but not in $S_2$. We use
$\interior(S)$ to denote the set of interior points of the set $S
\subseteq \reals^d$. A function $f:X \rightarrow \reals$ is locally
Lipschitz at $x \in X \subset \reals^d$ if there exists some
neighborhood $\mathcal{U}$ of $x$ and constant $K_x \ge 0$ such that
for all $x_1,x_2 \in \mathcal{U}$, it holds that $|f(x_1)-f(x_2)| \le
K_x \NormTwo{x_1-x_2}$. We say that $f$ is locally Lipschitz on $X$ if
it is locally Lipschitz at $x$ for all $x \in X$. The domain of $f$ is
denoted $\dom(f)$. The function $f$ is convex if for all $x_1,x_2 \in
X$ and all $\lambda \in [0,1]$, it holds that $f(\lambda x_1 +
(1-\lambda)x_2) \le \lambda x_1 + (1-\lambda)x_2$. Also, $f$ is
concave if $-f$ is convex. The generalized gradient of $f$ at $\hat{x}
\in X$ is defined as
\begin{align*}
  \partial f(\hat{x}) := \co \bigg \{ \lim_{i \rightarrow \infty} \nabla
  f(x_i) : x_i \rightarrow x, x_i \notin S \cup \Omega_f \bigg \},
\end{align*}
where $\co$ denotes the convex hull, $S$ is a set of measure zero, and
$\Omega_f \subset \reals^d$ is the set of points where $f$ is not
differentiable. If $f$ is locally Lipschitz, its generalized gradient
at any point in $X$ is non-empty. If $g:X \times Y \rightarrow
\reals$, then $\partial_xg(x,y)$ (resp. $\partial_yg(x,y)$) is the
generalized gradient of the map $x \mapsto g(x,y)$ (resp. $y \mapsto
g(x,y)$). For $c \in \reals$, we denote by $f^{-1}(\le c) = \setdef{x
  \in X}{f(x) \le c}$ the $c$-sublevel set of~$f$. Given $V:\reals^d
\rightarrow \reals$ and $f:\reals^d \rightarrow \reals^d$, the Lie
derivative of $V$ along $f$ at $x$ is
\begin{align}\label{eq:Lie}
  \mathcal{L}_f V(x) := \lim_{\alpha \rightarrow 0}\frac{V(x+\alpha
    f(x)) - V(x)}{\alpha}.
\end{align}
We say that $\mathcal{L}_fV(x)$ exists when the limit
in~\eqref{eq:Lie} exist. If $V$ is differentiable, then $\mathcal{L}_f
V(x) = \nabla V(x)^TF(x)$ for $x \in \reals^d$.

\subsection{Hybrid systems}\label{sec:hybrid-systems}

These basic notions on hybrid systems follow closely the exposition
found in~\cite{RG-RGS-ART:12}. A hybrid (or cyber-physical) system is
a dynamical system whose state may evolve according to (i) a
differential equation $\dot{x} = f(x)$ when its state is in some
subset, $C$, of the state-space and (ii) a difference equation $x^+ =
g(x)$ when its state is in some other subset, $D$, of the
state-space. Thus, we may represent a hybrid system by the tuple $H =
(f,g,C,D)$ where $f: \reals^d \rightarrow \reals^d$ (resp. $g:
\reals^d \rightarrow \reals^d$) is called the \emph{flow map}
(resp. \emph{jump map}) and $C \subseteq \reals^d$ (resp. $D \subseteq
\reals^d$) is called the \emph{flow set} (resp. \emph{jump
  set}). Formally speaking, the evolution of the states of $H$ are
governed by the following equations
\begin{subequations} \label{eq:H}
  \begin{align}
    \dot{x} &= f(x), \quad x \in C,
    \\
    x^+ &= g(x), \quad x \in D.
  \end{align}
\end{subequations}

A \emph{compact hybrid time domain} is a subset of $\realsnonnegative
\times \naturals$ of the form
\begin{align*}
  E = \cup_{j=0}^{J-1} ([t_j,t_{j+1}],j),
\end{align*}
for some finite sequence of times $0 = t_0 \le t_1 \le \dots \le
t_J$. It is a hybrid time domain if for all $(T,J) \in E$, $E \cap
([0,T]\times \{0,\dots,J\})$ is a compact hybrid time domain.  A
function $\psi$ is a solution to the hybrid system~\eqref{eq:H} if
\begin{enumerate}
\item for all $j \in \naturals$ such that $I^j := \{t:(t,j) \in
  \dom(\psi)\}$ has non-empty interior
  \begin{alignat*}{2}
    \psi(t,j) &\in C, \quad &&\forall t \in \interior(I^j),
    \\
    \dot{\psi}(t,j) &= f(\psi(t,j)), \quad &&\text{for almost all } t
    \in I^j.
  \end{alignat*}
\item for all $(t,j) \in \dom(\psi)$ such that $(t,j+1) \in
  \dom(\psi)$
  \begin{align*} \psi(t,j) &\in D,
    \\
    \psi(t,j+1) &= g(\psi(t,j)).
  \end{align*}
\end{enumerate}
In (i) above, we say that $\psi$ is \emph{flowing} and in (ii) we say
that $\psi$ is \emph{jumping}. We call $\psi$ \emph{persistently
  flowing} if it is eventually continuous or if there exists a uniform
time constant $\tau_P$ whereby $\psi$ flows for $\tau_P$ seconds
infinitely often. Formally speaking, $\psi$ is persistently flowing if
\begin{enumerate}
\item[(PFi)] $([t_J,\infty),J) \subset \dom(\psi)$ for some $J \in
  \naturals$, or
\item[(PFii)] there exists $\tau_P > 0$ and an increasing sequence
  $\{j_k\}_{k=0}^{\infty} \subset \naturals$ such that
  $([t_{j_k},t_{j_k} + \tau_P],j_k) \subset \dom(\psi)$ for each $k
  \in \naturals$.
\end{enumerate}

\subsection{Quadratic optimization}\label{sec:optimization}

Here we introduce some basic definitions and results regarding
mathematical optimization. A detailed exposition on these topics can
be found in~\cite{SB-LV:09}. First, a quadratic optimization problem
can be denoted by
\begin{subequations}\label{eq:opt}
  \begin{alignat}{2}
    &\min && \quad c^Tx + \frac{1}{2}x^T\mathcal{E}x
    \\
    &\hspace{1.5mm}\text{s.t.} && \quad Ax = b, \quad x \ge 0,
  \end{alignat}
\end{subequations}
where for $n,m \in \naturals$, $c,x \in \reals^n$, $0 \preceq
\mathcal{E} = \mathcal{E}^T \in \reals^{n \times n}$, $b \in
\reals^m$, and $A \in \reals^{m \times n}$. We call~\eqref{eq:opt} the
\emph{primal} problem and its associated \emph{dual} is defined as
\begin{align*}
  \max_{z} q(z),
\end{align*}
where $q:\reals^m \rightarrow \reals$ is given by
\begin{align*}
  q(z) := \min_x \bigg\{-\frac{1}{2}x^T\mathcal{E}x - b^Tz : c +
  \mathcal{E}x + A^Tz \ge 0 \bigg\}.
\end{align*} 
The solutions to the primal and the dual are related through the
so-called Karush-Kuhn-Tucker (KKT) conditions. A point $(x_*,z_*) \in
\reals^n \times \reals^m$ satisfies the KKT conditions
for~\eqref{eq:opt} if
\begin{gather*}
  c + \mathcal{E}x_* + A^Tz_* \ge 0, \quad Ax_* = b, \quad x_* \ge 0,
  \\
  (c + \mathcal{E}x_* + A^Tz_*)^Tx_* = 0.
\end{gather*}
When the primal is feasible with a finite optimal value,
\begin{enumerate}
\item a point $(x_*,z_*)$ satisfies the KKT conditions
  for~\eqref{eq:opt} if and only if $x_*$ (resp. $z_*$) is a
  solution to the primal (resp. the dual)
\item the optimal value of the primal is the optimal value
  of the dual.
\end{enumerate}

\myclearpage
\section{Problem statement and network
  model}\label{sec:problem-statement}

This section introduces the problem of interest.  Our main objective
is to develop a distributed algorithm for multi-agent systems that is
able to solve general linear programs and takes into account the
discrete nature of inter-agent communication. Given $c \in \reals^n$,
$b \in \reals^m$, and $A \in \reals^{m \times n}$, a linear program in
standard form on $\reals^n$ is defined by
\begin{subequations}\label{eq:linprog}
  \begin{alignat}{2}
    &\min && \quad c^Tx
    \\
    &\hspace{1.5mm}\text{s.t.} && \quad Ax = b, \quad x \ge 0.
  \end{alignat}
\end{subequations}
Note that this is a special case of the quadratic
program~\eqref{eq:opt}, where $\mathcal{E} = 0$.  We assume
that~\eqref{eq:linprog} is feasible, with a finite optimal value, and
denote by $\primalsol \subseteq \reals^n$ the set of
solutions. Without loss of generality, we assume that
\begin{align*}
  \text{{\bf SA \#1:}} \quad \rho(A^TA) \le 1.
\end{align*}
We do this for ease of presentation, as this assumption simplifies the
exposition of the technical treatment. Two reasons justify the
generality of \text{{\bf SA \#1}}: the results are easily extensible
to the case $\rho(A^TA) > 1$ and Remark~\ref{rem:da-SA1} later
presents a $O(m)$ distributed algorithm that a multi-agent network can
run to ensure the assumption holds.


We next describe the model for the multi-agent network. Consider a
collection of agents with unique identifiers $i \in
\{1,\dots,n\}$. The state of agent $i$ is $x_i \in \reals$. Each agent
$i$ knows $c_i$ and the non-zero elements of any row $a_\ell$ of $A$
where $a_{\ell,i} \not=0$ (and the associated $b_{\ell}$). In
addition, if $x_i$ and $x_j$ appear in a common constraint, then $i$
and $j$ have the ability to communicate with each other at discrete
instants of time of their choosing. We assume communication happens
instantaneously and denote by $\hat{x}_i$ the last state transmitted
by agent $i$ to its neighboring agents.  Our objective is then to
design a distributed algorithm that specifies \emph{how} agents should
update their own states with the information they possess and
\emph{when} they should broadcast it to neighboring agents with the
ultimate goal of making the aggregate of the agents' states $x =
(x_1,\dots,x_n)$ converge to a solution of~\eqref{eq:linprog}.  To
solve this problem, we take an approach based on continuous-time
computation with discrete-time communication. Formally, we formulate
our approach using the notion of hybrid system described in
Section~\ref{sec:hybrid-systems} as follows.

\begin{problemstatement}\longthmtitle{Distributed linear programming
    with event-triggered communication}\label{problem}
  Design a hybrid system that, for each $i \in \{1,\dots,n\}$, takes
  the form,
  \begin{subequations}\label{eq:problem}
    \begin{alignat}{2}
      \dot{x}_i &= g_i(\hat{x}), \quad &&\text{\rm{if $(x,\hat{x})
          \notin \dtriggerset$,}} \label{eq:problem_flow}
      \\
      \hat{x}_i^+ &= x_i, \quad &&\text{{\rm if $(x,\hat{x}) \in
          \dtriggerset$,}} \label{eq:problem_jump}
    \end{alignat}
  \end{subequations}
  where $g_i:\reals^n \rightarrow \reals$ is agent $i$'s flow map and
  $\dtriggerset \subseteq \reals^n \times \reals^n$ is agent $i$'s
  trigger set, which determines when $i$ should broadcast its state,
  such that
  \begin{enumerate}
  \item $g_i$ is computable by agent $i$ and the inclusion
    $(x,\hat{x}) \in \dtriggerset$ is detectable by agent $i$ using
    local information and communication, and
  \item the aggregate of the agents' states converge to a solution
    of~\eqref{eq:linprog}.
  \end{enumerate}
\end{problemstatement}

The interpretation of Problem~\ref{problem} is as follows.
Equation~\eqref{eq:problem_flow} models the fact that agent $i\in
\until{n}$ uses the last broadcast states from neighboring agents and
itself to compute the continuous-time flow $g_i$ governing the
evolution of its state. In-between two consecutive broadcasts of agent
$i$ (i.e., while flowing), there is no dynamics for its last broadcast
state $\hat{x}_i$. Formally, $\dot{\hat{x}}_i = 0$ if $(x,\hat{x})
\notin \dtriggerset$.  For this reason, the state evolution is quite
easy to compute since it changes according to a constant rate during
continuous flow. Our use of the term ``continuous-time flow'' is
motivated by the fact that we model the event-triggered design in the
hybrid system framework. Moreover, viewing the
dynamics~\eqref{eq:problem_flow} as a continuous-time flow will aid
our analysis in subsequent sections. Equation~\eqref{eq:problem_jump}
models the broadcast procedure. The condition $(x,\hat{x}) \in
\dtriggerset$ is a state-based trigger used by agent $i$ to determine
when to broadcast its current state $x_i$ to its neighbors. Since
communication is instantaneous, $x_i^+=x_i$ if $(x,\hat{x}) \in
\dtriggerset$.  The dynamical coupling between different agents is
through the broadcast states in $\hat{x}$ only. Note that an agent
cannot pre-determine the time of its next state broadcast because it
cannot predict if or when it will receive a broadcast from a
neighbor. For this reason, we call the strategy outlined in
Problem~\ref{problem} \emph{event-triggered} as opposed to
self-triggered.

When we do not specify the continuous-time dynamics of $\hat{x}_i$ the
reader should interpret this to mean that $\dot{\hat{x}}_i = 0$ when
$(x,\hat{x}) \notin \dtriggerset$. Likewise, $x_i^+ = x_i$ when
$(x,\hat{x}) \in \dtriggerset$. This convention holds throughout the
paper.

\myclearpage
\section{Continuous-time computation and
  communication}\label{sec:discontinuous-dynamics}

In this section we introduce a distributed continuous-time algorithm
that requires continuous communication to solve general linear
programs in standard form. We build on this design in the forthcoming
sections to provide an algorithmic solution to Problem~\ref{problem}
that only employs communication at discrete time instants.

We follow a design methodology similar to the one we employed in our
previous work~\cite{DR-JC:13-tac}. The reason for the different
dynamics proposed here has to do with its amenability to
event-triggered optimization and will become clearer in
Section~\ref{sec:centralized}.

\subsection{Solutions of linear program formulated as saddle
  points}\label{sec:q-reg}

The general approach we use for designing the continuous-time dynamics
is to define an augmented Lagrangian function based on a
regularization of the linear program and then derive the natural
saddle-point dynamics for that function. More specifically, consider
the following quadratic regularization of~\eqref{eq:linprog},
\begin{subequations}\label{eq:pert}
  \begin{alignat}{2}
    &\min && \quad \gamma c^Tx + \frac{1}{2} x^Tx
    \\ &\hspace{1.5mm}\text{s.t.} && \quad Ax = b, \quad x \ge 0.
  \end{alignat}
\end{subequations}
where $\gamma \ge 0$. We use the regularization rather than the
original formulation~\eqref{eq:linprog} itself because, as we discuss
later in Remark~\ref{re:motivation-qp}, the resulting saddle-point
dynamics is amenable to event-triggered implementation. The following
result reveals that this regularization is exact for suitable values
of~$\gamma$. The result is a modification of~\cite[Theorem
1]{OLM-RRM:79} for linear programs in standard form,, rather than in
inequality form.

\begin{lemma}\longthmtitle{Exact regularization}\label{lem:pert}
  There exists $\gamma_{\min} > 0$ such that, for $\gamma \ge
  \gamma_{\min}$, the solution to the regularization~\eqref{eq:pert}
  is a solution to the linear program~\eqref{eq:linprog}.
\end{lemma}
\begin{proof}
  We use the fact that a point $x_* \in \reals^n$ (resp. $z_* \in
  \reals^m$) is a solution to~\eqref{eq:linprog} (resp. the dual
  of~\eqref{eq:linprog}) if and only if it satisfies the KKT
  conditions for~\eqref{eq:linprog},
  \begin{align}\label{eq:linprog_KKT} 
    c + A^Tz_* \ge 0, \quad A x_* = b, \quad x_* \ge 0, \quad (c +
    A^Tz_*)^Tx_* = 0 .
  \end{align}
  We also consider the optimization problem
  \begin{subequations}\label{eq:pert_proof}
    \begin{alignat}{2}
      &\min && \quad \frac{1}{2} x^Tx \\ &\hspace{1.5mm}\text{s.t.} &&
      \quad Ax = b, \quad c^Tx = p, \quad x \ge 0,
    \end{alignat}
  \end{subequations}
  where $p$ is the optimal value of~\eqref{eq:linprog}. Note that the
  solution to the above problem is a solution to~\eqref{eq:linprog} by
  construction of the constraints. Likewise,
  $(\bar{x},\bar{z},\bar{w}) \in \reals^n \times \reals^m \times
  \reals$ are primal-dual solutions to~\eqref{eq:pert_proof} if and
  only if they satisfy the KKT conditions for~\eqref{eq:pert_proof}
  \begin{align}
      \bar{x} + A^T \bar{z} + c \bar{w} \ge 0, \quad A \bar{x} = b,
      \quad c^T\bar{x} = p, \quad \bar{x} \ge 0, \quad  (\bar{x} + A^T
      \bar{z} + c \bar{w})^T\bar{x} = 0. \label{eq:pert_proof_KKT}
  \end{align}
  Since $\bar{x}$ is a solution to~\eqref{eq:linprog}, without loss of
  generality, we suppose that $x_* = \bar{x}$. We consider the cases
  when (i) $\bar{w} = 0$, (ii) $\bar{w} > 0$, and (iii) $\bar{w} <
  0$. In case (i), combining~\eqref{eq:linprog_KKT}
  and~\eqref{eq:pert_proof_KKT}, one can obtain for any $\gamma \ge
  0$,
  \begin{align*}
    \gamma c + x_* + A^T(\gamma z_* + \bar{z}) \ge
    0, \quad A x_* = b, \quad x_* \ge 0, \quad (\gamma c + x_* +
    A^T(\gamma z_* + \bar{z}))^Tx_* = 0.
  \end{align*}
  The above conditions reveal that $(x_*,\gamma z_* + \bar{z})$
  satisfy the KKT conditions for~\eqref{eq:pert}. Thus, $x_*$ (which
  is a solution to~\eqref{eq:linprog}) is the solution
  to~\eqref{eq:pert} and this would complete the proof. Next, consider
  case (ii). If $\gamma = \gamma_{\min} := \bar{w} > 0$, the
  conditions~\eqref{eq:pert_proof_KKT} can be manipulated to give
  \begin{align*}
    \gamma c + x_* + A^T\bar{z} \ge 0, \quad A x_* = b,
    \quad x_* \ge 0, \quad (\gamma c + x_* +
    A^T\bar{z})^Tx_* = 0.
  \end{align*}
  This means that $(x_*,\bar{z})$ satisfy the KKT conditions
  for~\eqref{eq:pert}. Thus, $x_*$ (which is a solution
  to~\eqref{eq:linprog}) is the solution to~\eqref{eq:pert} and this
  would complete the proof. Now, for any $\gamma \ge \gamma_{\min}$,
  there exists an $\eta \ge 0$ such that $\gamma = \gamma_{\min} +
  \eta = \bar{w} + \eta$. Combining~\eqref{eq:linprog_KKT}
  and~\eqref{eq:pert_proof_KKT}, one can obtain
  \begin{gather*}
    \gamma c + x_* + A^T(\eta z_* + \bar{z}) \ge 0, \quad A x_* = b,
    \quad x_* \ge 0
    \\
    (\gamma c + x_* + A^T(\eta z_* + \bar{z})^Tx_* = 0.
  \end{gather*}
  This means that $(x_*,\eta z_* + \bar{z})$ satisfy the KKT
  conditions for~\eqref{eq:pert}. Thus, $x_*$ (which is a solution
  to~\eqref{eq:linprog}) is the solution to~\eqref{eq:pert} and this
  would complete the proof. Case (iii) can be considered analogously
  to case (ii) with $\gamma_{\min} := -\bar{w}$. This completes the
  proof.
\end{proof}

The actual value of $\gamma_{\min}$ in Lemma~\ref{lem:pert} is
somewhat generic in the following sense: if one replaces $c$
in~\eqref{eq:pert} by $\bar{\gamma} c$ for some $\bar{\gamma} > 0$,
then the regularization is exact for $\gamma \ge
\frac{\gamma_{\min}}{\bar{\gamma}}$. Therefore, to ease notation, we
make the following standing assumption,
\begin{align*}
  \text{{\bf SA \#2:}} \quad \gamma_{\min} \le 1.
\end{align*}
This justifies our focus on the case $\gamma = 1$.  Our next result
establishes a correspondence between the solution of~\eqref{eq:pert}
and the saddle points of an associated augmented Lagrangian function.

\begin{lemma}\longthmtitle{Solutions of~\eqref{eq:pert} as saddle
    points}\label{lem:lagrangian}
  For $K \ge 0$, consider the augmented Lagrangian function
  $\map{L^K}{\reals^n \times \reals^m}{\reals}$ associated to the
  optimization problem~\eqref{eq:pert} with $\gamma =1$,
  \begin{align*}
    L^{K}(x,z) &= c^Tx + \frac{1}{2} x^Tx +
    \frac{1}{2}(Ax-b)^T(Ax-b)  + z^T(Ax-b) + K
    \mathbbm{1}^T \max \{0,-x\}.
  \end{align*}
  Then, $L^{K}$ is convex in $x$ and concave in $z$. Let $x_* \in
  \reals^n$ (resp. $z_* \in \reals^m$) be the solution
  to~\eqref{eq:pert} (resp. a solution to the dual
  of~\eqref{eq:pert}). Then, for $K > \NormInf{c+ x_* + A^Tz_*}$, the
  following holds,
  \begin{enumerate}
  \item $(x_*,z_*)$ is a saddle point of $L^{K}$,
  \item if $(\bar{x},\bar{z})$ is a saddle point of $L^{K}$, then
    $\bar{x} = x_*$ is the solution of~\eqref{eq:pert}.
  \end{enumerate}
\end{lemma}
\begin{proof}
  For any $x \in \reals^n$,
  \begin{align}\label{eq:lagrangian_bound}
    L^{K}(x,z_*) &= c^Tx + \frac{1}{2} x^Tx + \frac{1}{2}(Ax-b)^T(Ax-b) +
    z_*^T(Ax-b) + K \mathbbm{1}^T
    \max\{0,-x\} \notag
    \\
    &\ge c^Tx + \frac{1}{2} x^Tx + z_*^T(Ax-b) + \NormInf{c+ x_* +
      A^Tz_*} x \nonumber
    \\
    &\ge c^Tx + \frac{1}{2} x^Tx + z_*^T(Ax-b) - (c+ x_* + A^Tz_*)^T x
    \nonumber
    \\
    &\ge c^Tx + \frac{1}{2} x^Tx + z_*^TA(x-x_*) - (c+ x_* + A^Tz_*)^T
    (x - x_*) \nonumber
    \\
    &\ge c^Tx + \frac{1}{2} x^Tx - (c + x_*)^T(x-x_*) \nonumber
    \\
    &\ge c^Tx_* + \frac{1}{2} (x-x_*)^T(x-x_*) + \frac{1}{2} x_*^Tx_*
    \nonumber
    \\
    &\ge c^Tx_* + \frac{1}{2} x_*^Tx_* =
    L^{K}(x_*,z_*).
  \end{align}
  For any $z$, it is easy to see that $L^{K}(x_*,z) =
  L^{K}(x_*,z_*)$. Thus $(x_*,z_*)$ is a saddle point of~$L^{K}$.

  Let us now prove (ii). As a necessary condition for
  $(\bar{x},\bar{z})$ to be a saddle point of $L^{K}$, it must be that
  $\partial_zL^{K}(\bar{x},\bar{z}) = A\bar{x}-b = 0$ as well as
  $L^{K}(x_*,\bar{x}) \ge L^{K}(\bar{x},\bar{z})$ which means that
  \begin{align}\label{eq:saddle_point_contradiction}
    &c^Tx_* +
    \frac{1}{2} x_*^Tx_* \ge c^T \bar{x} + \frac{1}{2}
    \bar{x}^T\bar{x} + K \mathbbm{1}^T\max\{0,-\bar{x}\} .
  \end{align}
  If $\bar{x} \ge 0$ then $c^Tx_* + \frac{1}{2} x_*^Tx_* \ge
  c^T\bar{x} + \frac{1}{2} \bar{x}^T\bar{x} $ and thus $\bar{x}$ would
  be a solution to~\eqref{eq:pert}. Consider then that $\bar{x} \not
  \ge 0$. Then,
  \begin{align*}
    c^T\bar{x} + \frac{1}{2} \bar{x}^T\bar{x} &= c^Tx_* + \frac{1}{2}
    x_*^Tx_* + (c+x_*)^T(\bar{x}-x_*) + \frac{1}{2} (\bar{x} -
    x_*)^T(\bar{x} - x_*)
    \\
    & \ge c^Tx_* + \frac{1}{2} x_*^Tx_* + (c + x_*)^T(\bar{x}-x_*)
    \\
    & \ge c^Tx_* + \frac{1}{2} x_*^Tx_* - z_*^TA(\bar{x}-x_*) + (c +
    x_* + A^Tz_*)^T(\bar{x}-x_*)
    \\
    & \ge c^Tx_* + \frac{1}{2} x_*^Tx_* - z_*^T(A\bar{x}-b) + (c + x_*
    + A^Tz_*)^T \bar{x}
    \\
    & \ge c^Tx_* + \frac{1}{2} x_*^Tx_* - \NormInf{c + x_* + A^Tz_*}
    \max\{0,-\bar{x}\}
    \\
    & > c^Tx_* + \frac{1}{2} x_*^Tx_* - K
    \mathbbm{1}^T\max\{0,-\bar{x}\} ,
  \end{align*}
  contradicting~\eqref{eq:saddle_point_contradiction}. Thus, $\bar{x}
  \ge 0$ and must be the solution to~\eqref{eq:pert}. 
\end{proof}

\subsection{Distributed continuous-time dynamics}

Given the results in Lemmas~\ref{lem:pert} and~\ref{lem:lagrangian}, a
sensible strategy to find a solution of~\eqref{eq:linprog} is to
employ the saddle-point dynamics associated to the augmented
Lagrangian function~$L^K$. Formally, we set
\begin{subequations}\label{eq:saddle-point-dyn}
  \begin{align}
    \dot{x} &\in -\partial_xL^{K}(x,z),
    \\
    \dot{z} &=
    \partial_zL^{K}(x,z).
  \end{align}
\end{subequations}
This dynamics is well-defined since $L^K$ is a locally Lipschitz
function. For an appropriate choice of the parameter $K$, one can
establish the asymptotic convergence of trajectories of the
saddle-point dynamics to a point in the set $\primalsol \times
\reals^m$. However, we build the event-triggered implementation on a
different dynamics that does not require precise knowledge of $K$. The
following result reveals that one can characterize certain elements of
the saddle-point dynamics, regardless of $K$, which will allow us to
design a discontinuous dynamics later.
 
\begin{lemma}\longthmtitle{Generalized gradients of the
    Lagrangian}\label{lem:grad}
  Let $\map{\nomflow}{\reals^n \times \reals^m}{\reals^n}$ be defined
  by $ \nomflow(x,z) = -(A^Tz+c + x)-A^T(Ax-b)$. Given a compact set
  $X \times Z \subset \reals^n_{\ge 0} \times \reals^m$, let
  \begin{align*}
    \Kbound (X \times Z) := \max_{(x,z) \in X \times Z}
    \|\nomflow(x,z)\|_{\infty}.
  \end{align*}
  Then, if $ K \ge \Kbound (X \times Z)$, for each $(x,z) \in X \times
  Z$, $\partial_zL^{K}(x,z) = \{Ax - b\}$ and there exists $a \in
  -\partial_xL^{K}(x,z) \subset \reals^n$ such that, for each $i \in
  \{1,\dots,n\}$,
  \begin{alignat*}{2}
    a_i &= \begin{cases} \nomflow_i(x,z), \hspace{14.2mm}\text{{\rm
          if} $x_i > 0$} ,
      \\
      \max\{0,\nomflow_i(x,z)\}, \; \text{{\rm if} $x_i = 0$} .
    \end{cases}
  \end{alignat*}
\end{lemma}
\begin{proof}
  Let $(x,z) \in \reals^n \times \reals^m $. It is straightforward to
  see that $ \partial_zL^{K}(x,z) = \{ Ax-b \}$ for any $K$.  Next,
  note that for any $a \in -\partial_xL^{K}(x,z)$, we have
  \begin{align}\label{eq:subset}
    -a - (A^Tz + c + x) - A^T(Ax-b) \in K \partial \max \{0,-x\} ,
  \end{align}
  or, equivalently, $-a + \nomflow(x,z) \in K \partial \max \{0,-x\}$.
  For any $i \in \{1,\dots,n\}$ such that $x_i > 0$, the corresponding
  set in the right-hand side of~\eqref{eq:subset} is the singleton $0$
  and therefore $ a_i = \nomflow_i(x,z)$.  On the other hand, if $x_i
  = 0$, then
  \begin{align*}
    -a_i + \nomflow_i(x,z) \in [-K,0].
  \end{align*}
  If $\nomflow_i(x,z) \ge 0$, the choice $a_i = \nomflow_i(x,z)$
  satisfies the equation. Conversely, if $\nomflow_i(x,z) < 0$, then
  $a_i = 0$ satisfies the equation for all $K \ge \Kbound(X \times Z)$
  by definition of~$\Kbound(X \times Z)$.
  This completes the proof.
\end{proof}

As suggested previously, the above result enables us to propose an
alternative continuous-time (discontinuous) dynamics to
solve~\eqref{eq:linprog} that does not require knowledge
of~$K$. Specifically, Lemma~\ref{lem:grad} implies that, on a compact
set, the trajectories of the dynamics
\begin{subequations}\label{eq:disc-dyn}
  \begin{align}
    \dot{x}_i &=
    \begin{cases}
      \nomflow_i(x,z), & \text{if $x_i > 0$} ,
      \\
      \max\{0,\nomflow_i(x,z)\}, &  \text{if $x_i = 0$} ,
    \end{cases}
    \label{eq:x-dyn}
    \\
    \dot{z} &= Ax-b ,
    \label{eq:z-dyn}
  \end{align}
\end{subequations}
with $ i \in \{1,\dots,n\}$, are trajectories
of~\eqref{eq:saddle-point-dyn}. That is, any bounded trajectory
of~\eqref{eq:disc-dyn} is also a trajectory
of~\eqref{eq:saddle-point-dyn}. As a consequence, the convergence
properties of the saddle-point dynamics are also inherited
by~\eqref{eq:disc-dyn} and this is the dynamics that we design an
event-triggered implementation for in the next~sections. Besides the
standard considerations in designing an event-triggered implementation
(such as ensuring convergence and preventing arbitrarily fast
broadcasting), we face several unique challenges including the fact
that the equilibria of~\eqref{eq:disc-dyn} are not known a priori as
well as having to account for the switched nature of the dynamics.

We conclude this section by discussing the distributed implementation
of~\eqref{eq:disc-dyn} based on the network model of
Section~\ref{sec:problem-statement}. 

\begin{remark}\longthmtitle{Distributed implementation via virtual
    agents}\label{rem:distributed-implementation}
  {\rm Note that the dynamics~\eqref{eq:disc-dyn} possess auxiliary
    variables $z \in \reals^m$ corresponding to the Lagrange
    multipliers of the optimization problem. For the purpose of
    analysis, we consider $m$ virtual agents with identifiers
    $\{n+1,\dots,n+m\}$, where virtual agent $n+\ell$ updates the
    state $z_{\ell} \in \reals$ and has knowledge of the data
    $a_{\ell} \in \reals^n $ and $b_{\ell} \in \reals$. It is
    important to observe that, in an actual implementation, the state
    and dynamics of virtual agent $n+\ell$ can be stored and
    implemented by any of the real agents with knowledge of $a_{\ell}$
    and $b_{\ell}$. For each $\ell \in \{1,\dots,m\}$, the set of
    agents
    \begin{align*}
      \{ i \in \{1,\dots,n\} : a_{\ell,i} \not=0 \} \cup
      \{ n+\ell \},
    \end{align*}
    can communicate their state information to each other. In words,
    if $x_i$ and $x_j$ appear in constraint $\ell$, then agents $i,j$,
    and $n+\ell$ can communicate with each other.  The \emph{neighbor}
    set of~$i$, denoted $\neigh_i$, is the set of all agents that $i$
    can communicate with. The set of real (resp. virtual) neighbors of
    $i$ is $\neigh_i^x := \neigh_i \cap \{1,\dots,n\}$
    (resp. $\neigh_i^z := \neigh_i \cap \{n+1,\dots,n+m\}$). Under
    these assumptions, it is straightforward to verify that agent $i
    \in \{1,\dots,n\}$ can compute $\nomflow_i(x,z)$ using local
    information and can thus implement~\eqref{eq:x-dyn}. Likewise, a
    virtual agent $n+\ell \in \{n+1,\dots,n+m\}$ can compute and
    implement its corresponding dynamics $\dot{z}_{\ell} =
    a_{\ell}^Tx-b_{\ell}$ in~\eqref{eq:z-dyn}. The network structure
    described here is quite natural for many real-world network
    optimization problems that can be formulated as linear programs.}
  \oprocend
\end{remark}

\begin{remark}\longthmtitle{Distributed algorithm to enforce {\bf{SA
        \#1}}}\label{rem:da-SA1}
  {\rm Given the model described in
    Remark~\ref{rem:distributed-implementation}, we explain briefly
    here how agents can implement a simple pre-processing algorithm
    based on $\max$-consensus to ensure that $\rho(A^TA) \le 1$. For
    each row $a_{\ell}$ of the matrix $A$, the virtual agent $n +
    \ell$ can compute the $\ell^{\operatorname{th}}$ row of
    $A^TA$. Then, this agent stores the following estimate of the
    spectral radius,
    \begin{align*}
      \hat{\rho}_{n + \ell} = (A^TA)_{(\ell,\ell)} + \sum_{\substack{i
          \in \{1,\dots,n\} \setminus \{\ell\}}} |(A^TA)_{(\ell,i)}| .
    \end{align*}
    The virtual agents use these estimates as an initial point in the
    standard $\max$-consensus algorithm~\cite{JC:08-auto}.  In $O(m)$
    steps, the $\max$-consensus converges to a point $\rho_* \ge
    \rho(A^TA)$, where the inequality is a consequence of the
    Gershgorin Circle Theorem~\cite[Corollary 6.1.5]{RAH-CRJ:85}.
    Then, each virtual agent scales its corresponding row of $A$ and
    entry of $b$ by $1 / \rho_*$, and communicates this new data to
    its neighbors. The resulting linear program is $\min \{c^Tx :
    \tilde{A}x = \tilde{b}, \; x \ge 0\}$, with $\tilde{A} = A /
    \rho_*$ and $\tilde{b} = b / \rho_*$. Both the solutions and
    optimal value of the new linear program are the same as the
    original linear program and, by construction,
    $\rho(\tilde{A}^T\tilde{A}) \le 1$. \oprocend}
\end{remark}

\myclearpage
\section{Algorithm design with centralized event-triggered
  communication}\label{sec:centralized}

Here, we build on the discussion of
Section~\ref{sec:discontinuous-dynamics} to address the main objective
of the paper as outlined in Problem~\ref{problem}. Our starting point
is the distributed continuous-time algorithm~\eqref{eq:disc-dyn},
which requires continuous-time communication.  Our approach is divided
in two steps because of the complexity of the challenges (e.g.,
asymptotic convergence, asynchronism, and Zeno behavior) involved in
going from continuous-time to opportunistic discrete-time
communication.  In this section, we design a centralized
event-triggered scheme that the network can use to determine in an
opportunistic way when information should be updated. In the next
section, we build on this development to design a distributed
event-triggered communication scheme that individual agents can employ
to determine when to share information with their neighbors.

The problem we solve in this section can be formally stated as
follows.

\begin{problemstatement}\longthmtitle{Linear programming with
    centralized event-triggered communication}\label{problem2}
  Identify a set $\ctriggerset \subseteq \reals^n_{\ge 0} \times
  \reals^m \times \reals^n_{\ge 0} \times \reals^m$ such that the
  hybrid system that, for $i \in \{1,\dots,n\}$, takes the form,
  \begin{subequations}\label{eq:hybrid2}
    \begin{align}
      \dot{x}_i &=
      \begin{cases}
        \nomflow_i(\hat{x},\hat{z}), &\hat{x}_i > 0,
        \\
        \max \{0,\nomflow_i(\hat{x},\hat{z})\}, &\hat{x}_i =
        0, \end{cases}
       \label{eq:hybrid2_contx} 
       \\
       \dot{z} &= A\hat{x}-b, \label{eq:hybrid2_contz}
     \end{align}
     if $(x,z,\hat{x},\hat{z}) \not \in \ctriggerset$ and
    \begin{align}
      (\hat{x}^+,\hat{z}^+) = (x,z),  \label{eq:hybrid2_disc}
    \end{align}
  \end{subequations}
  if $(x,z,\hat{x},\hat{z}) \in \ctriggerset$, makes the aggregate $x
  \in \reals^n$ of the real agents' states converge to a solution of
  the linear program~\eqref{eq:linprog}.
\end{problemstatement}

We refer to the set $\ctriggerset$ in Problem~\ref{problem2} as the
\emph{centralized trigger set}. Note that, in this centralized
formulation of the problem, we do not require individual agents, but
rather the network as a whole, to be able to detect whether
$(x,z,\hat{x},\hat{z}) \in \ctriggerset$. In addition, when this
condition is enabled, state broadcasts among agents are performed
synchronously, as described by~\eqref{eq:hybrid2_disc}. Our strategy
to design $\ctriggerset$ is to first identify a candidate Lyapunov
function and study its evolution along the trajectories
of~\eqref{eq:hybrid2}. We then synthesize $\ctriggerset$ based on the
requirement that our Lyapunov function decreases along the
trajectories of~\eqref{eq:hybrid2} and conclude with a result showing
that the desired convergence properties are attained.

Before we introduce the candidate Lyapunov function, we present an
alternative representation
of~\eqref{eq:hybrid2_contx}-\eqref{eq:hybrid2_contz} that will be
useful in our analysis later. Given $(\hat{x},\hat{z}) \in
\realsnonnegative \times \reals^m$, let $\sigma(\hat{x},\hat{z})$ be
the set of agents $i$ for which $\dot{x}_i =
\nomflow_i(\hat{x},\hat{z})$ in~\eqref{eq:hybrid2_contx}. Formally,
\begin{align*}
  \sigma(\hat{x},\hat{z}) = \big\{i \in \{1,\dots,n\}:
  \nomflow_i(\hat{x},\hat{z}) \ge 0 \; \text{or} \; \hat{x}_i > 0
  \big\}.
\end{align*}
Next, let  $I_{\sigma(\hat{x},\hat{z})} \in \reals^{n \times n}$ be
defined by
\begin{align*}
  (I_{\sigma(\hat{x},\hat{z})})_{i,j} =
  \begin{cases}
    0, \; \text{if $i \not= j$ or $i \notin \sigma(\hat{x},\hat{z})$},
    \\
    1, \; \text{otherwise.}
  \end{cases}
\end{align*}
Note that this matrix is an identity-like matrix with a zero
$(i,i)$-element if $i \notin \sigma(\hat{x},\hat{z})$.  The matrix
$I_{\sigma(\hat{x},\hat{z})}$ has the following properties,
\begin{align*}
  I_{\sigma(\hat{x},\hat{z})} \succeq 0, \quad
  I_{\sigma(\hat{x},\hat{z})} = I_{\sigma(\hat{x},\hat{z})}^T, \quad
  I_{\sigma(\hat{x},\hat{z})}^2 = I_{\sigma(\hat{x},\hat{z})}, \quad
  \rho(I_{\sigma(\hat{x},\hat{z})}) \le 1 .
\end{align*}
Then, a compact representation
of~\eqref{eq:hybrid2_contx}-\eqref{eq:hybrid2_contz} is
\begin{align*}
  (\dot{x},\dot{z}) &= F(\hat{x},\hat{z}) :=
  (I_{\sigma(\hat{x},\hat{z})} \nomflow(\hat{x},\hat{z}),A\hat{x}-b),
\end{align*}
where $F = (F_x,F_z) : \realsnonnegative^n \times \reals^m \rightarrow
\reals^n \times \reals^m$.

\subsection{Candidate Lyapunov function and its evolution}

Now let us define and analyze the candidate Lyapunov function that we
use to design the trigger set~$\ctriggerset$.  The overall Lyapunov
function is the sum of $2$ separate functions $V_1$ and $V_2$, that we
introduce next. To define $V_1 : \reals^n \times \reals^m \rightarrow
\reals_{\ge 0}$, fix $\mathcal{K} > \NormInf{c+ x_* + A^Tz_*}$ where
$x_*$ (resp. $z_*$) is the solution to~\eqref{eq:pert} (resp. any
solution of the dual of~\eqref{eq:pert}) and let $(\bar{x},\bar{z})$
be a saddle-point of $L^{\mathcal{K}}$. Then
\begin{align*}
  V_1(x,z) &= \frac{1}{2}(x-\bar{x})^T(x-\bar{x}) +
  \frac{1}{2}(z-\bar{z})^T(z-\bar{z}).
\end{align*}
Note that $V_1 \ge 0$ is smooth with compact sublevel sets.
Next, $V_2 : \reals^n \times \reals^m \rightarrow \reals_{\ge 0}$ is
given by
\begin{align*}
  V_2(x,z) &= \frac{1}{2} \nomflow(x,z)^TI_{\sigma(x,z)} \nomflow(x,z)
  + \frac{1}{2}(Ax-b)^T(Ax-b).
\end{align*}
Note that $V_2\ge 0$ but, due to the state-dependent matrix
$I_{\sigma(x,z)}$, is only piecewise smooth. In this sense $V_2$ can
be viewed as a collection of multiple (smooth) Lyapunov functions,
each defined on a neighborhood where $\sigma$ is constant. Also,
$V_2^{-1}(0)$ is, by definition, the set of saddle-points of
$L^{\mathcal{K}}$ (cf. Lemma~\ref{lem:grad}). It turns out that the
negative terms in the Lie derivative of $V_1$ alone are insufficient
to ensure that $V_1$ is always decreasing given any practically
implementable trigger design (by \emph{practically implementable} we
mean a trigger design that does not demand arbitrarily fast state
broadcasting). The analogous statement regarding $V_2$ is also true
which is why we consider instead a candidate Lyapunov function $V :
\reals^n \times \reals^m \rightarrow \reals_{\ge 0}$ that is their sum
\begin{align*}
  V(x,z) = (V_1 + V_2)(x,z).
\end{align*}
The following result states an upper bound on $\mathcal{L}_FV$ in
terms of the state errors in $x$ and~$z$. 

\begin{proposition}\longthmtitle{Evolution of $V$}\label{prop:Lie-V}
  Let $X \times Z \subseteq \realsnonnegative^n \times \reals^m$ be
  compact and suppose that $(x,z,\hat{x},\hat{z}) \in X \times Z
  \times X \times Z$ is such that $\sigma(\hat{x},\hat{z}) \subseteq
  \sigma (x,z)$ and
  \begin{align} \label{eq:lim_sig}
    \sigma(x,z) = \lim_{\alpha \rightarrow 0}\sigma(x + \alpha
    F_x(\hat{x},\hat{z}), z + \alpha F_z(\hat{x},\hat{z})).
  \end{align} 
  Then $\mathcal{L}_FV(x,z)$ exists and
  \begin{align}
    \mathcal{L}_FV(x,z) &\le
    -\frac{1}{2}\nomflow(\hat{x},\hat{z})^TI_{\sigma(\hat{x},\hat{z})}
    \nomflow(\hat{x},\hat{z}) -\frac{1}{4}(A\hat{x}-b)^T(A\hat{x}-b) +
    40 e_x^Te_x + 20 e_z^Te_z \nonumber
    \\
    &\quad + 15\nomflow(x,z)^TI_{\sigma(x,z) \setminus
      \sigma(\hat{x},\hat{z})}\nomflow(x,z) , \label{eq:Lie-V}
  \end{align}
  where $ e_x = x-\hat{x}$ and $e_z = z-\hat{z}$.
\end{proposition}
\begin{proof}
  For convenience, we use the shorthand notation $p = \sigma (x,z)$
  and $\hat{p} = \sigma (\hat{x},\hat{z})$. Consider first $V_1$,
  which is differentiable and thus $\mathcal{L}_FV_1(x,z)$ exists,
  \begin{align}
    \mathcal{L}_FV_1(x,z) &=
    (x-\bar{x})^TI_{\hat{p}}\nomflow(\hat{x},\hat{z}) +
    (z-\bar{z})^T(A\hat{x}-b) \nonumber
    \\
    &= (\hat{x}-\bar{x})^TI_{\hat{p}}\nomflow(\hat{x},\hat{z}) +
    (\hat{z}-\bar{z})^T(A\hat{x}-b) +
    e_x^TI_{\hat{p}}\nomflow(\hat{x},\hat{z}) +
    e_z^T(A\hat{x}-b). \label{eq:V_1-1}
  \end{align}
  Since $X \times Z$ is compact, without loss of generality assume
  that $\mathcal{K} \ge \Kbound(X \times Z)$ so that
  $-I_{\hat{p}}\nomflow(\hat{x},\hat{z}) \in \partial_x
  L^{\mathcal{K}}(\hat{x},\hat{z})$, cf. Lemma~\ref{lem:grad}.  This,
  together with the fact that $L^{\mathcal{K}}$ is convex in its first
  argument, implies
  \begin{align*}
    L^{\mathcal{K}}(\hat{x},\hat{z}) \le
    L^{\mathcal{K}}(\bar{x},\hat{z}) - (\hat{x}-\bar{x})^T
    I_{\hat{p}}\nomflow(\hat{x},\hat{z}).
  \end{align*}
  Since $L^{\mathcal{K}}$ is linear in $z$, we can write
  \begin{align*}
    L^{\mathcal{K}}(\hat{x},\hat{z}) =
    L^{\mathcal{K}}(\hat{x},\bar{z}) + (\hat{z}-\bar{z})^T(A\hat{x}-b)
    .
  \end{align*}
  Substituting these expressions into~\eqref{eq:V_1-1}, we get
  \begin{align*}
    \mathcal{L}_FV_1(x,z) &\le
    L^{\mathcal{K}}(\bar{x},\hat{z}) -
    L^{\mathcal{K}}(\hat{x},\bar{z}) +
    e_x^TI_{\hat{p}}\nomflow(\hat{x},\hat{z}) + e_z^T(A\hat{x}-b)
    \\
    & \le c^T\bar{x} + \frac{1}{2}\bar{x}^T\bar{x} - c^T\hat{x} -
    \frac{1}{2}\sum_{i=1}^n\hat{x}^T\hat{x} -
    \bar{z}^T(A\hat{x}-b)-\mathcal{K}\mathbbm{1}^T\max\{0,-\hat{x}\}
    \\&\quad - \frac{1}{2}(A\hat{x}-b)^T(A\hat{x}-b) +
    e_x^TI_{\hat{p}}\nomflow(\hat{x},\hat{z}) + e_z^T(A\hat{x}-b).
  \end{align*}
  From the analysis in the proof of Lemma~\ref{lem:lagrangian},
  inequality~\eqref{eq:lagrangian_bound} showed that
  \begin{align*}
    &c^T\bar{x} + \frac{1}{2}\bar{x}^T\bar{x}
    \\
    & \quad \le c^T\hat{x} + \frac{1}{2}\hat{x}^T\hat{x} +
    \bar{z}^T(A\hat{x}-b) + \frac{1}{2}(A\hat{x}-b)^T(A\hat{x}-b) +
    \mathcal{K}\mathbbm{1}^T\max\{0,-\hat{x}\} ,
  \end{align*}
  where we use the fact that $\bar{x}$ is also a solution
  to~\eqref{eq:pert}, cf. Lemma~\ref{lem:lagrangian}. Therefore,
  \begin{align}
    \mathcal{L}_FV_1(x,z) &\le
    -\frac{1}{2}(A\hat{x}-b)^T(A\hat{x}-b) +
    e_x^TI_{\hat{p}}\nomflow(\hat{x},\hat{z}) + e_z^T(A\hat{x}-b)
    \nonumber
    \\
    &\le -\frac{1}{2}(A\hat{x}-b)^T(A\hat{x}-b) +
    \frac{\kappa}{2}(A\hat{x}-b)^T(A\hat{x}-b) \nonumber
    \\
    &\quad + \frac{\kappa}{2}
    \nomflow(\hat{x},\hat{z})^TI_{\hat{p}}\nomflow(\hat{x},\hat{z}) +
    \frac{1}{2\kappa}e_x^Te_x + \frac{1}{2\kappa}e_z^Te_z
    , \label{eq:V_1_bound}
  \end{align}
  where we have used Lemma~\ref{lem:bilinear}. Next, let us consider
  $V_2$. We begin by showing that~\eqref{eq:lim_sig} is sufficient for
  $\mathcal{L}_FV_2(x,z)$ to exist. Since $\sigma$ is a discrete set
  of indices, for the limit in~\eqref{eq:lim_sig} to exist, there must
  exist an $\bar{\alpha} > 0$ such that
  \begin{align*}
    \sigma(x,z) = \sigma(x + \alpha
    F_x(\hat{x},\hat{z}), z + \alpha F_z(\hat{x},\hat{z})),
  \end{align*}
  for all $\alpha \in [0,\bar{\alpha}]$. This means that one can
  substitute $I_{\sigma(x,z)}$ for $I_{\sigma(x + \alpha
    F_x(\hat{x},\hat{z}), z + \alpha F_z(\hat{x},\hat{z}))}$ in the
  definition of the Lie derivative~\eqref{eq:Lie}. Since
  $I_{\sigma(x,z)}$ is constant with respect to $\alpha$, it is
  straightforward to see that
  \begin{align*}
    \mathcal{L}_FV_2(x,z) &= \frac{1}{2}\nabla( \nomflow(x,z)^TI_p
    \nomflow(x,z))^TF(\hat{x},\hat{z}) + \frac{1}{2}
    \nabla((Ax-b)^T(Ax-b))^TF(\hat{x},\hat{z}) .
  \end{align*}
  Thus,
  \begin{align}
    \mathcal{L}_FV_2(x,z) &= \nomflow(x,z)^TI_p (D_x \nomflow(x,z)
    F_x(\hat{x},\hat{z}) + D_z \nomflow(x,z) F_z(\hat{x},\hat{z})) +
    (Ax-b)^T A F_x(\hat{x},\hat{z}) \nonumber
    \\
    &= -\nomflow(x,z)^TI_p(A^TA +
    I)I_{\hat{p}}\nomflow(\hat{x},\hat{z}) -
    \nomflow(x,z)^TI_pA^T(A\hat{x}-b) \nonumber
    \\
    &\quad +
    (Ax-b)^TAI_{\hat{p}}\nomflow(\hat{x},\hat{z}) \label{eq:V_2-dot}.
  \end{align}
  Due to the assumption that $\hat{p} \subseteq p$, we can write $I_p
  = I_{\hat{p}} + I_{p \setminus \hat{p}}$. Also, $\nomflow(x,z)$ can
  be written equivalently in terms of the errors $e_x,e_z$ as
  \begin{align*}
    \nomflow(x,z) = \nomflow(\hat{x},\hat{z}) - A^Te_z - e_x - A^TAe_x .
  \end{align*} 
  Likewise, $Ax-b = A\hat{x} - b + Ae_x$. Substituting these
  quantities into~\eqref{eq:V_2-dot},
  \begin{align}
    \mathcal{L}_FV_2(x,z) &= -(\nomflow(\hat{x},\hat{z}) - A^Te_z -
    e_x - A^TAe_x)^TI_{\hat{p}}(A^TA + I
    )I_{\hat{p}}\nomflow(\hat{x},\hat{z}) \nonumber
    \\
    &\quad - \nomflow(x,z)^TI_{p \setminus \hat{p}}(A^TA +
    I)I_{\hat{p}}\nomflow(\hat{x},\hat{z}) \nonumber
    \\
    &\quad - (\nomflow(\hat{x},\hat{z}) - A^Te_z - e_x -
    A^TAe_x)^TI_{\hat{p}}A^T(A\hat{x}-b) \nonumber
    \\
    &\quad - \nomflow(x,z)^TI_{p \setminus \hat{p}}A^T(A\hat{x}-b) +
    (A\hat{x} - b +
    Ae_x)^TAI_{\hat{p}}\nomflow(\hat{x},\hat{z}). \label{eq:V_2_bound_prelim}
  \end{align}
  We now derive upper bounds for a few terms
  in~\eqref{eq:V_2_bound_prelim}. For example,
  \begin{align*}
    e_z^TA I_{\hat{p}}A^TAI_{\hat{p}}\nomflow(\hat{x},\hat{z}) &\le
    \frac{1}{2\kappa}e_z^Te_z +
    \frac{\kappa}{2}\nomflow(\hat{x},\hat{z})^T
    I_{\hat{p}}A^TAI_{\hat{p}}A^TA
    I_{\hat{p}}A^TAI_{\hat{p}}\nomflow(\hat{x},\hat{z})
    \\
    & \le \frac{1}{2\kappa}e_z^Te_z + \frac{\kappa}{2}
    \nomflow(\hat{x},\hat{z})^TI_{\hat{p}}\nomflow(\hat{x},\hat{z}),
  \end{align*}
  where we have used Lemma~\ref{lem:bilinear} and
  Theorem~\ref{th:interlacing} along with the facts that $\rho(A^TA) =
  \rho(AA^T) \le 1$ and $\rho(I_{\hat{p}}) \le 1$. Likewise,
  \begin{align*}
    e_x^T I_{\hat{p}}A^TAI_{\hat{p}}\nomflow(\hat{x},\hat{z}) &\le
    \frac{1}{2\kappa}e_x^Te_x + \frac{\kappa}{2}
    \nomflow(\hat{x},\hat{z})^TI_{\hat{p}}A^TA
    I_{\hat{p}}A^TAI_{\hat{p}}\nomflow(\hat{x},\hat{z})
    \\
    & \le \frac{1}{2\kappa}e_x^Te_x + \frac{\kappa}{2}
    \nomflow(\hat{x},\hat{z})^TI_{\hat{p}}\nomflow(\hat{x},\hat{z}),
  \end{align*}
  and 
  \begin{align*}
    \nomflow(x,z)^TI_{p \setminus \hat{p}}A^T(A\hat{x}-b) &\le
    \frac{1}{2\kappa} \nomflow(x,z)^TI_{p \setminus
      \hat{p}}\nomflow(x,z) + \frac{\kappa}{2}
    (A\hat{x}-b)^TAA^T(A\hat{x}-b)
    \\
    &\le \frac{1}{2\kappa} \nomflow(x,z)^TI_{p \setminus
      \hat{p}}\nomflow(x,z) + \frac{\kappa}{2}
    (A\hat{x}-b)^T(A\hat{x}-b).
  \end{align*}
  Repeatedly bounding every term in~\eqref{eq:V_2_bound_prelim} in an
  analogous way (which we omit for the sake of space and presentation)
  and adding the bound~\eqref{eq:V_1_bound}, we attain the following
  inequality
  \begin{align*}
    \mathcal{L}_FV(x,z) &\le
    -(1-5\kappa)\nomflow(\hat{x},\hat{z})^TI_{\hat{p}}
    \nomflow(\hat{x},\hat{z}) -
    \frac{1}{2}(1-5\kappa)(A\hat{x}-b)^T(A\hat{x}-b)
    \\
    &\quad + \frac{3}{2\kappa}
    \nomflow(x,z)^TI_{p\setminus\hat{p}}\nomflow(x,z) +
    \frac{4}{\kappa} e_x^Te_x + \frac{2}{\kappa}e_z^Te_z .
  \end{align*}
  Equation~\eqref{eq:Lie-V} follows by choosing $\kappa =
  \frac{1}{10}$, completing the proof.
\end{proof}

The reason why we have only considered the case
$\sigma(\hat{x},\hat{z}) \subseteq \sigma(x,z)$ (and not the more
general case of $\sigma(\hat{x},\hat{z}) \not= \sigma(x,z)$) when
deriving the bound~\eqref{eq:Lie-V} in Proposition~\ref{prop:Lie-V} is
the following: our distributed trigger design later (specifically, the
trigger sets $\dtriggersetzero$ introduced in
Section~\ref{sec:distributed}) ensures that $\sigma(\hat{x},\hat{z})
\subseteq \sigma(x,z)$ always. For this reason, we need not know how
$V$ evolves in the more general case.

\subsection{Centralized trigger set design and convergence
  analysis}\label{sec:centralized-trigger-design}

Here, we use our knowledge of the evolution of the function $V$,
cf. Proposition~\ref{prop:Lie-V}, to design the centralized trigger
set $\ctriggerset$. Our approach is to incrementally design subsets of
$\ctriggerset$ and then combine them at the end to define
$\ctriggerset$. The main observation that we base our design on is the
following: The first two terms in the right-hand-side
of~\eqref{eq:Lie-V} are negative and thus desirable and the rest are
positive. However, following a state broadcast, the positive terms
become zero. This motivates our first trigger set that should belong
to $\ctriggerset$,
\begin{multline}\label{eq:ctriggersete}
  \ctriggersete := \{ (x,z,\hat{x},\hat{z}) \in (\realsnonnegative^n
  \times \reals^m)^2 : \;A\hat{x} - b \not= 0 \text{ or }
  I_{\sigma(\hat{x},\hat{z})}\nomflow(\hat{x},\hat{z}) \not= 0, \text{
    and} 
  \\
  \frac{1}{8}(A\hat{x}-b)^T(A\hat{x}-b) +
  \frac{1}{4}\nomflow(\hat{x},\hat{z})^TI_{\sigma(\hat{x},\hat{z})}
  \nomflow(\hat{x},\hat{z}) \le 20 e_z^Te_z + 40e_x^Te_x \}.
\end{multline}
The numbers $\tfrac{1}{8}$ and $\tfrac{1}{4}$ in the inequalities that
define $\ctriggersete$ are design choices that we have made to ease
the presentation. Any other choice in $(0,1)$ is also possible, with
the appropriate modifications in the ensuing exposition.  Note that,
when both $A\hat{x} - b =
I_{\sigma(\hat{x},\hat{z})}\nomflow(\hat{x},\hat{z}) = 0$, no state
broadcasts are required since the system is at a (desired)
equilibrium.

Likewise, after a state broadcast, $\sigma(x,z) =
\sigma(\hat{x},\hat{z})$ and $I_{\sigma(x,z) \setminus
  \sigma(\hat{x},\hat{z})} = 0$ which means that the last term
in~\eqref{eq:Lie-V} is also zero. For this reason, define
\begin{align}\label{eq:ctriggersetsigma}
  \ctriggersetsigma := \{ (x,z,\hat{x},\hat{z}) \in
  (\realsnonnegative^n \times \reals^m)^2 :\sigma(x,z) \not=
  \sigma(\hat{x},\hat{z}) \},
\end{align}
which prescribes a state broadcast when the mode $\sigma$ changes.

We require one final trigger for the following reason. While the set
$\realsnonnegative^n \times \reals^m$ is invariant under the
continuous-time dynamics~\eqref{eq:disc-dyn}, this does not hold any
more in the event-triggered case because agents use outdated state
information. To preserve the invariance of this set, we define
\begin{align}\label{eq:ctriggersetzero}
  \ctriggersetzero & := \{(x,z,\hat{x},\hat{z}) \in
  (\realsnonnegative^n \times \reals^m)^2 : \exists i \in \{1,\dots,
  n\} \text{ s.t. } \hat{x}_i > 0, \, x_i = 0 \}.
\end{align}
If this trigger is activated by some agent $i$'s state becoming zero,
then it is easy to see from the definition of the
dynamics~\eqref{eq:hybrid2_contx} that $\dot{x}_i \ge 0$ after the
state broadcast and thus $x_i$ remains non-negative. Finally, the
overall centralized trigger set~is
\begin{align}\label{eq:h_centralized}
  \ctriggerset := \ctriggersete \cup \ctriggersetsigma \cup
  \ctriggersetzero.
\end{align} 
The following result characterizes the convergence properties
of~\eqref{eq:hybrid2} under the centralized event-triggered
communication scheme specified by~\eqref{eq:h_centralized}.

\begin{theorem}\longthmtitle{Convergence of the centralized
    event-triggered design}\label{th:convergence_prelim}
  If $\psi$ is a persistently flowing solution of~\eqref{eq:hybrid2}
  with $\ctriggerset$ defined as in~\eqref{eq:h_centralized}, then
  there exists a point $(x_*,z') \in \primalsol \times \reals^m$ such
  that,
  \begin{align*}
    \psi(t,j) \rightarrow (x_*,z',x_*,z') \quad \text{as} \quad t+j
    \longrightarrow \infty, \quad (t,j) \in \dom(\psi) .
  \end{align*}
\end{theorem}
\begin{proof}
  Let $(t,j) \mapsto \psi (t,j) = (x (t,j),z (t,j),\hat{x}
  (t,j),\hat{z} (t,j))$. We begin the proof by showing that $V$ is
  non-increasing along~$\psi$. To this end, it suffices to prove that
  (a) $\mathcal{L}_FV(x(t,j),z(t,j)) \le 0$ when $\psi$ is flowing and
  $\mathcal{L}_FV(x(t,j),z(t,j))$ exists, (b) $V(x(t,j),z(t,j)) \le
  \lim_{\tau \rightarrow t^-}V(x(t,j),z(t,j))$ when $\psi(t,j)$ is
  flowing but $\mathcal{L}_FV(x(t,j),z(t,j))$ does not exist, and (c)
  $V(x(t,j+1),z(t,j+1)) \le V(x(t,j),z(t,j))$ when $\psi$ is jumping.

  We begin with (a) and consider $(t,j) \in \dom(\psi)$ for which
  $\psi$ is flowing. Then, the interval $I^{j} := \{t:(t,j) \in
  \dom(\psi)\}$ has non-empty interior and $t \in
  \interior(I^{j})$. This means that $\psi(t,j) \notin \ctriggerset$
  and, in particular, $\sigma(x(t,j),z(t,j)) =
  \sigma(\hat{x}(t,j),\hat{z}(t,j))$ by construction of
  $\ctriggersetsigma$. Also, let $X \times Z$ be a compact set
  such that $\psi(t,j) \in X \times Z \times X \times Z$. Therefore,
  the conditions of Proposition~\ref{prop:Lie-V} are satisfied and
  $\mathcal{L}_FV(x(\tau,j),z(\tau,j))$ exists for all $\tau \in
  \interior(I^j)$. Using~\eqref{eq:Lie-V}, it holds that
  \begin{align*}
    \mathcal{L}_FV(x(t,j),z(t,j)) &\le
    -\frac{1}{8}(A\hat{x}(t,j)-b)^T(A\hat{x}(t,j)-b)
    \\
    &\quad - \frac{1}{4}
    \nomflow(\hat{x}(t,j),\hat{z}(t,j))^TI_{\sigma(\hat{x}(t,j),\hat{z}(t,j))}
    \nomflow(\hat{x}(t,j),\hat{z}(t,j)),
  \end{align*}
  where we have used (i) the fact that, since $\sigma(x(t,j),z(t,j)) =
  \sigma(\hat{x}(t,j),\hat{z}(t,j))$, the last quantity
  in~\eqref{eq:Lie-V} is zero and (ii) the bound on $20e_x^Te_x +
  40e_z^Te_z$ in the definition of $\ctriggersete$. Clearly, in
  this case, $\mathcal{L}_FV(x(t,j),z(t,j)) \le 0$ when $\psi(t,j) \in
  X \times Z \times X \times Z$.

  Next, consider (b). Since $V_1$ is smooth,
  $\mathcal{L}_FV_1(x(t,j),z(t,j))$ exists, however, when
  $V_2(x(t,j),z(t,j))$ is discontinuous,
  $\mathcal{L}_FV(x(t,j),z(t,j))$ does not. This happen at any $(t,j)
  \in \dom(\psi)$ for which (i) $I^j$ (as defined previously) has
  non-empty interior and (ii) $\sigma(x(t,j),z(t,j)) \not= \lim_{\tau
    \rightarrow t^-}\sigma(x(\tau,j),z(\tau,j))$
  (cf. Proposition~\ref{prop:Lie-V}). Note that condition (i) ensures
  that the limit in condition (ii) is well-defined. For purposes of
  presentation, define the sets
  \begin{align*}
    \mathcal{S}_{+} := \sigma(x(t,j),z(t,j)) \setminus
    \text{$\textstyle \lim_{\tau \rightarrow t^-}$}
    \sigma(x(\tau,j),z(\tau,j)),
    \\
    \mathcal{S}_{-} := \text{$\textstyle \lim_{\tau \rightarrow t^-}$}
    \sigma(x(\tau,j),z(\tau,j)) \setminus \sigma(x(t,j),z(t,j)).
  \end{align*}
  Note that one of $\mathcal{S}_{+},\mathcal{S}_{-}$ may be empty. We
  can write
  \begin{multline*}
    V_2(x(t,j),z(t,j)) = \text{$\textstyle \lim_{\tau \rightarrow
        \tau^-}$} V_2(x(\tau,j),z(\tau,j))
    \\
    + \frac{1}{2} \text{$\textstyle \sum_{i \in \mathcal{S}_{+}}$}
    \nomflow_i(x(t,j),z(t,j))^2 - \frac{1}{2} \text{$\textstyle
      \sum_{i \in \mathcal{S}_{-}}$} \nomflow_i(x(t,j),z(t,j))^2.
  \end{multline*}
  For each $i \in \mathcal{S}_+$, it must be that
  $\nomflow_i(x(t,j),z(t,j)) = 0$ since
  $\nomflow_i(\hat{x}(t,j),\hat{z}(t,j)) < 0$ and $\nomflow,x,z$ are
  continuous. Moreover, the last term in the right-hand-side of the
  above expression is non-positive. Thus, $V_2(x(t,j),z(t,j)) \le
  \lim_{\tau \rightarrow t^-} V_2(x(\tau,j),z(\tau,j))$.

  Next, when $\psi$ is jumping, as is case (c), $V(x(t,j),z(t,j)) =
  V(x(t,j+1),z(t,j+1)$ because $(x(t,j+1),z(t,j+1)) = (x(t,j),z(t,j))$
  according to~\eqref{eq:hybrid2_disc}.

  To summarize, $V(x(t,j),z(t,j))$ is non-increasing when $\psi(t,j)
  \in X \times Z \times X \times Z$. Without loss of generality, we
  choose $X \times Z = V^{-1}(\le c)$, where $c =
  V(x(0,0),z(0,0))$. $X \times Z$ is compact because the sublevel sets
  of $V_1$ are compact, and $X \times Z \times X \times Z$ is
  invariant so as not to contradict $V(x(t,j),z(t,j))$ being
  non-increasing on $X \times Z \times X \times Z$. Thus,
  $V(x(t,j),z(t,j))$ is non-increasing at all times and $\psi$ is
  bounded.

  Now we establish the convergence property
  of~\eqref{eq:hybrid2}. First note that, in this preliminary design,
  $\psi$ being persistently flowing implies also that the Lie
  derivative of $V$ along $F$ exists for $\tau_P$ time on those
  intervals of persistent flow. This is because $\psi$ flowing implies
  that $\sigma$ is constant and thus the Lie derivative of $V$ along
  $F$ exists (cf.~\eqref{eq:lim_sig}). There are two possible
  characterizations of persistently flowing $\psi$ as given in
  Section~\ref{sec:prelim}. Consider (PFi). By the boundedness of
  $\psi$ just established, it must be that for all $t \ge t_J$
  \begin{alignat*}{2}
    0 &= \dot{x}(t,j) &&=
    I_{\sigma(\hat{x}(t_J,J),\hat{z}(t_J,J))}\nomflow(\hat{x}(t_J,J)),\hat{z}(t_J,J)),
    \\
    0 &= \dot{z}(t,j) &&= A\hat{x}(t_J,J)-b.
  \end{alignat*}
  By Lemma~\ref{lem:grad} this means that
  $(\hat{x}(t_J,J),\hat{z}(t_J,J))$ is a saddle-point of
  $L^{\mathcal{K}}$ (without loss of generality, we assume that
  $\mathcal{K} \ge K_*(X \times Z)$). Applying
  Lemma~\ref{lem:lagrangian} reveals that $\hat{x}(t_J,J) \in
  \primalsol$. Since $\hat{x}(t,j)$ is a sampled version of $x(t_J,J)$
  it is clear that $x(t_J,J) \in \primalsol$ as well, and since their
  dynamics are stationary in finite time, they converge to a point,
  completing the proof.

  Consider then (PFii), the second characterization of persistently
  flowing. We have established that
  $\{V(x(t_{j_k},j_k),z(t_{j_k},j_k))\}_{k = 0}^{\infty}$ is
  non-increasing. Since it is also bounded from below by 0, by the
  monotone convergence theorem there exists a $V_* \in [0,c]$ such
  that $\lim_{k \rightarrow \infty}V(x(t_{j_k},j_k),z(t_{j_k},j_k)) =
  V_*$. Thus
  \begin{align*}
    V(x(t_{j_k},j_k),z(t_{j_k},j_k)) -
  V(x(t_{j_{k+1}},j_{k+1}),z(t_{j_{k+1}},j_{k+1})) \rightarrow 0.
  \end{align*}
  Let $\delta > 0$ and consider $\kappa \in \naturals$ such that
  \begin{align*}
    V(x(t_{j_k},j_k),z(t_{j_k},j_k)) -
    V(x(t_{j_{k+1}},j_{k+1}),z(t_{j_{k+1}},j_{k+1})) < \delta ,
  \end{align*}
  for all $k \ge \kappa$. By the bound established on
  $\mathcal{L}_FV(x(t,j_k),z(t,j_k))$, which exists for all $(t,j_k)
  \in ([t_{j_k},t_{j_k}+\tau_P),j_k)$, it holds that,
  \begin{align*}
    V(x(t_{j_{k+1}},&j_{k+1}),z(t_{j_{k+1}},j_{k+1}))
    \\
    &\le V(x(t_{j_k},j_k),z(t_{j_k},j_k)) -
    \frac{1}{8}(A\hat{x}(t_{j_k},j_k)-b)^T(A\hat{x}(t_{j_k},j_k)-b)\tau_P
    \\
    & \quad - \frac{1}{4}
    \nomflow(\hat{x}(t_{j_k},j_k),\hat{z}(t_{j_k},j_k))^T
    I_{\sigma(\hat{x}(t_{j_k},j_k),\hat{z}(t_{j_k},j_k))}
    \nomflow(\hat{x}(t_{j_k},j_k),\hat{z}(t_{j_k},j_k))\tau_P .
  \end{align*}
  Therefore, $V(x(t_{j_k},j_k),z(t_{j_k},j_k)) -
  V(x(t_{j_{k+1}},j_{k+1}),z(t_{j_{k+1}},j_{k+1})) < \delta$ for all
  $k \ge \kappa$ implies that
  \begin{align*}
    \nomflow(\hat{x}(t_{j_k},j_k),\hat{z}(t_{j_k},j_k))^T
    I_{\sigma(\hat{x}(t_{j_k},j_k),\hat{z}(t_{j_k},j_k))}
    \nomflow(\hat{x}(t_{j_k},j_k),\hat{z}(t_{j_k},j_k)) & \le 4\delta
    \tau_P ,
    \\
    (A\hat{x}(t_{j_k},j_k)-b)^T(A\hat{x}(t_{j_k},j_k)-b) & \le 8\delta
    \tau_P,
  \end{align*}
  for all $k \ge \kappa$. Since $\tau_P$ is a uniform constant and
  $\delta > 0$ can be taken arbitrarily small, we deduce
  \begin{align*}
    I_{\sigma(\hat{x}(t_{j_k},j_k),\hat{z}(t_{j_k},j_k))}
    \nomflow(\hat{x}(t_{j_k},j_k),\hat{z}(t_{j_k},j_k)) \rightarrow 0
    \quad \text{and} \quad A\hat{x}(t_{j_k},j_k) - b \rightarrow 0,
  \end{align*}
  as $k \rightarrow \infty$. By Lemma~\ref{lem:grad}, this means that
  $(\hat{x}(t_{j_k},j_k),\hat{z}(t_{j_k},j_k))$ converges to the set
  of saddle-points of $L^{\mathcal{K}}$. The same argument holds for
  $x(t_{j_k},j_k)$ since $\hat{x}(t_{j_k},j_k)$ is a sampled version
  of that state.

  Finally, we establish the convergence to a point. By the
  Bolzano-Weierstrass Theorem, there exists a subsequence
  $\{j_{k_{\ell}}\}$ such that
  $(x(t_{j_{k_{\ell}}},j_{k_{\ell}}),z(t_{j_{k_{\ell}}},j_{k_{\ell}}))$
  converges to a saddle-point $(\bar{x}',\bar{z}')$ of
  $L^{\mathcal{K}}$. Fix $\delta' > 0$ and let $\ell_*$ be such that
  \begin{align*}
    \NormTwo{(x(t_{j_{k_{\ell}}},j_{k_{\ell}}),z(t_{j_{k_{\ell}}},j_{k_{\ell}}))
      - (\bar{x}',\bar{z}')} < \delta',
  \end{align*}
  for all $\ell \ge \ell_*$. Consider the function $W = W_1 + V_2$
  where
  \begin{align*}
    W_1(x,z) = \frac{1}{2}(x-\bar{x}')^T(x-\bar{x}') +
    \frac{1}{2}(z-\bar{z}')^T(z-\bar{z}').
  \end{align*}
  Let $c' =
  W(x(t_{j_{k_{\ell_*}}},j_{k_{\ell_*}}),z(t_{j_{k_{\ell_*}}},j_{k_{\ell_*}}))$
  and $X' \times Z' = W^{-1}(\le c')$. Repeating the previous
  analysis, but for $W$ instead of $V$, we deduce that $X' \times Z'
  \times X' \times Z'$ is invariant. Consequently,
  $\NormTwo{(x(t,j),z(t,j)) - (\bar{x}',\bar{z}')} < \delta'$ for all
  $(t,j) \in \dom(\psi)$ such that $t + j \ge t_{j_{k_{\ell_*}}} +
  j_{k_{\ell_*}}$. Since $\delta' > 0$ is arbitrary, it
  holds that
  \begin{align*}
    \psi(t,j) \rightarrow (\bar{x}',\bar{z}',\bar{x}',\bar{z}') \quad
    \text{as} \quad t + j \longrightarrow \infty, \quad (t,j) \in
    \dom(\psi) .
  \end{align*}
  Since $(\bar{x}',\bar{z}')$ is a saddle-point of $L^{\mathcal{K}}$
  and, without loss of generality $\mathcal{K} \ge K_*(X \times Z)$,
  applying Lemma~\ref{lem:lagrangian} reveals that $\bar{x}' \in
  \primalsol$, which completes the proof.
\end{proof}

\begin{remark}\longthmtitle{Motivation for  quadratic regularization
    of 
    linear program}\label{re:motivation-qp}
  {\rm Here we revisit the claim made in
    Section~\ref{sec:discontinuous-dynamics} that using the
    saddle-point dynamics derived for the original linear
    program~\eqref{eq:linprog} would not be amenable to an
    event-triggered implementation. If we were to follow the same
    design methodology for such dynamics, we would find that the bound
    on the Lie derivative of $V$ would resemble~\eqref{eq:Lie-V}, but
    without the non-positive term $-\nomflow(\hat{x},\hat{z})^T
    I_{\sigma(\hat{x},\hat{z})}\nomflow(\hat{x},\hat{z})$. Following
    the same methodology to identify the trigger set, one would then
    use the trigger
    \begin{align*}
      \frac{1}{8}(A\hat{x}-b)^T(A\hat{x}-b) \le 20 e_z^Te_z +
      40e_x^Te_x ,
    \end{align*}
    to define~$\ctriggersete$ and ensure that the function~$V$
    does not increase. However, this trigger may easily result in
    continuous-time communication: consider a scenario where
    $A\hat{x}-b = 0$, but the state~$x$ is still evolving. Then the
    trigger would require continuous-time broadcasting of~$x$ to
    ensure that $e_x$ remains zero.   }\oprocend
\end{remark}

\myclearpage
\section{Algorithm design with distributed event-triggered
  communication}\label{sec:distributed}

In this section, we provide a distributed solution to
Problem~\ref{problem}, e.g., a coordination algorithm to solve linear
programs requiring only communication at discrete instants of time
triggered by criteria that agents can evaluate with local
information. Our strategy to accomplish this is to investigate to what
extent the centralized triggers identified in
Section~\ref{sec:centralized-trigger-design} can be implemented in a
distributed way. In turn, making these triggers distributed poses the
additional challenge of dealing with the asynchronism in the state
broadcasts across different agents, which raises the possibility of
non-persistency in the solutions. We deal with both issues in our
forthcoming discussion and establish the convergence of our
distributed design.

\subsection{Distributed trigger set
  design}\label{sec:distributed-trigger-design}

Here, we design distributed triggers that individual agents can
evaluate with the local information available to them to guarantee the
monotonically decreasing evolution of the candidate Lyapunov
function~$V$. Our design methodology builds on the centralized trigger
sets~$\ctriggersete$, $\ctriggersetsigma$, and $\ctriggersetzero$ of
Section~\ref{sec:centralized-trigger-design}. As a technical detail,
the distributed algorithm that results from this section has an
extended state which, for ease of notation, we denote by
\begin{align*}
  \xi = (x,z,s,q,r,\hat{x},\hat{z}) \subseteq \Xi := \reals^n_{\ge 0}
  \times \reals^m \times \reals^n_{\ge 0} \times \{0,1\}^{(n+m)\times n} \times
  \reals^{n+m} \times \reals^n_{\ge 0} \times \reals^m.
\end{align*}
The meaning and dynamics of states $s,q,$ and $r$ will be revealed as
they become necessary in our development.

We start by showing how the inequality that defines whether the
network state belongs to the set $\ctriggersete$
in~\eqref{eq:ctriggersete} can be distributed across the group of
agents. Given $\mu_1,\dots,\mu_{n+m} >0$, consider the following
trigger set for each agent,
\begin{align*}
  \dtriggersete :=
  \begin{cases}
    \{ \xi \in \Xi: \nomflow_i(\hat{x},\hat{z}) \not= 0 \text{ and }
    (e_x)_i^2 \ge \mu_i \nomflow_i(\hat{x},\hat{z})^2 \}, \;
    &\text{{\rm if $i \le n$,}}
    \\
    \{ \xi \in \Xi: a_{i-n}^T\hat{x} - b_{i-n} \not= 0 \text{ and }
    (e_z)_{i-n}^2 \ge \mu_i(a_{i-n}^T\hat{x} - b_{i-n})^2\}, \!
    &\text{{\rm if $i \ge n+1$.}}
  \end{cases}
\end{align*}
If each $\mu_i \le \tfrac{1}{160}$ and $(x,z,\hat{x},\hat{z})$ is such
that the inequalities defining each $\dtriggersete$ do not hold, then
it is clear that $(x,z,\hat{x},\hat{z}) \notin \ctriggersete$. To
ensure convergence of the resulting algorithm, we later characterize
the specific ranges for the design parameters $\{\mu_i\}_{i=1}^{n+m}$.

Next, we show how the inclusion of the network state in the triggered
set~$\ctriggersetzero$ defined in~\eqref{eq:ctriggersetzero} can be
easily evaluated by individual agents with partial information. In
fact, for each $i \in \until{n}$, define the set
\begin{align*}
  \dtriggersetzero := \{\xi \in \Xi :\hat{x}_i > 0 \text{ but } x_i = 0\} .
\end{align*}
Clearly, $(x,z,\hat{x},\hat{z}) \in \ctriggersetzero$ if and only if
there is $i \in \until{n}$ such that $\xi \in \dtriggersetzero$.

The triggered set $ \ctriggersetsigma$ defined
in~\eqref{eq:ctriggersetsigma} presents a greater challenge from a
distributed computation viewpoint. The problem is that, in the absence
of fully up-to-date information from its neighbors, an agent will fail
to detect the mode switches that characterize the definition of this
set.  The specific scenario we refer to is the following: assume agent
$i \in \until{n}$ has $x_i = 0$ and the information available to it
confirms that its state should remain constant, i.e., with
$\nomflow_i(\hat{x},\hat{z}) < 0$. If the condition $\nomflow_i(x,z)
\ge 0$ becomes true as the network state evolves, this fact is
undetectable by $i$ with its outdated information. In such a case, $i
\notin \sigma(\hat{x},\hat{z})$ but $i \in \sigma(x,z)$, meaning that
the equality $\sigma(x,z) = \sigma(\hat{x},\hat{z})$ defining the
trigger set $ \ctriggersetsigma$ would not be not enforced.  To deal
with this issue, we first need to understand the effect that a
mismatch in the modes has on the evolution of the candidate Lyapunov
function~$V$. We address this in the following result.

\begin{proposition}\longthmtitle{Bound on evolution of candidate
    Lyapunov function under mode mismatch}\label{prop:mismatch}
  Suppose that $(\hat{x},\hat{z}) \in \realsnonnegative^n \times
  \reals^m$ is such that $i \notin \sigma(\hat{x},\hat{z})$ for some
  $i \in \{1,\dots,n\}$ and let $t \mapsto (x(t),z(t))$ be the
  solution to
  \begin{align*}
    (\dot{x},\dot{z}) &= F(\hat{x},\hat{z}),
  \end{align*}
  starting from $(\hat{x},\hat{z})$. Let $T > 0$ be the minimum time
  such that $i \in \sigma(x(T),z(T))$. Then, for any $\nu > 0$, and
  all $t$ such that $t - T < \frac{\nu}{2\sqrt{2}}$, the following
  holds,
  \begin{align*}
    \nomflow_i(x(t),z(t))^2 \le \nu^2
    \nomflow(\hat{x},\hat{z})^TI_{\sigma(\hat{x},\hat{z}) \cap \neigh^x_i}
    \nomflow(\hat{x},\hat{z}) + \nu^2
    (A\hat{x}-b)^TI_{\neigh^z_i}(A\hat{x}-b).
  \end{align*} 
\end{proposition}
\begin{proof}
  We use the shorthand notation $\hat{p} = \sigma(\hat{x},\hat{z})$
  and $p(t) = \sigma(x(t),z(t))$. Since $i \notin \hat{p}$, it must be
  that $\hat{x}_i = 0$ and $\nomflow_i(\hat{x},\hat{z}) <
  0$. Moreover, if $i \in p(T)$, it must be, by continuity of $t
  \mapsto (x(t),z(t))$ and $(x,z) \mapsto \nomflow(x,z)$, that
  $\nomflow_i(x(T),z(T)) = 0$. Let us compute the Taylor expansion of
  $t \mapsto \nomflow_i(x(t),z(t))$ using $t=T$ as the initial
  point. For technical reasons, we actually consider the equivalent
  mapping $t \mapsto I_{\{i\}} \nomflow(x(t),z(t))$ instead,
  \begin{align*}
    I_{\{i\}} \nomflow(x(t),z(t)) & = I_{\{i\}}\nomflow(x(T),z(T)) +
    D_x I_{\{i\}}\nomflow(x(T),z(T))^TF_x(\hat{x},\hat{z})(t-T)
    \\
    & \quad + D_z
    I_{\{i\}}\nomflow(x(T),z(T))^TF_z(\hat{x},\hat{z})(t-T)
    \\
    &= I_{\{i\}}(A^TA+ I)I_{\hat{p}}\nomflow(\hat{x},\hat{z})(t-T) +
    I_{\{i\}}A^T(A\hat{x}-b)(t-T) ,
  \end{align*}
  where the equality holds because the higher order terms are
  zero. Thus,
  \begin{align*}
    \nomflow(x(t),z(t))^TI_{\{i\}} \nomflow(x(t),z(t)) &=
    \nomflow(\hat{x},\hat{z})^TI_{\hat{p}}(A^TA+ I)I_{\{i\}}(A^TA+
    I)I_{\hat{p}}\nomflow(\hat{x},\hat{z})(t-T)^2
    \\
    & \quad + 2 \nomflow(\hat{x},\hat{z})^TI_{\hat{p}}(A^TA+
    I)I_{\{i\}}A^T(A\hat{x}-b)(t-T)^2
    \\
    &\quad + (A\hat{x}-b)^TAI_{\{i\}}A^T(A\hat{x}-b)(t-T)^2 .
  \end{align*}
  Using Lemma~\ref{lem:bilinear} with $\kappa = \frac{1}{2}$ and
  exploiting the consistency between the matrix $A$ and the neighbors
  of $i$, we obtain
  \begin{align*}
    \nomflow(x(t)&,z(t))^TI_{\{i\}} \nomflow(x(t),z(t))
    \\
    &\le 2\nomflow(\hat{x},\hat{z})^TI_{\hat{p}}(A^TA+
    I)I_{\{i\}}(A^TA+ I)I_{\hat{p}}\nomflow(\hat{x},\hat{z})(t-T)^2
    \\
    &\quad + 2(A\hat{x}-b)^TAI_{\{i\}}A^T(A\hat{x}-b)(t-T)^2
    \\
    &\le 2\nomflow(\hat{x},\hat{z})^TI_{\hat{p}\cap
      \neigh^x_i}(A^TA+I)^2I_{\hat{p}\cap
      \neigh^x_i}\nomflow(\hat{x},\hat{z})(t-T)^2
    \\
    &\quad +
    2(A\hat{x}-b)^TI_{\neigh^z_i}AA^TI_{\neigh^z_i}(A\hat{x}-b)(t-T)^2
    \\
    &\le 8\nomflow(\hat{x},\hat{z})^TI_{\hat{p}\cap
      \neigh^x_i}\nomflow(\hat{x},\hat{z})(t-T)^2 +
    2(A\hat{x}-b)^TI_{\neigh^z_i}(A\hat{x}-b)(t-T)^2
    \\
    &\le 8(t-T)^2(\nomflow(\hat{x},\hat{z})^TI_{\hat{p}\cap
      \neigh^x_i}\nomflow(\hat{x},\hat{z}) +
    (A\hat{x}-b)^TI_{\neigh^z_i}(A\hat{x}-b)) .
  \end{align*}
  Using the bound $t-T \le \frac{\nu}{2\sqrt{2}}$ in the statement of
  the result completes the proof.
\end{proof}

The importance of Proposition~\ref{prop:mismatch} comes from the
following observation: given the upper bound on the evolution of the
candidate Lyapunov function~$V$ obtained in
Proposition~\ref{prop:Lie-V}, one can appropriately choose the value
of $\nu$ so that the negative terms in~\eqref{eq:Lie-V} can compensate
for the presence of the last term due to a mode mismatch of finite
time length.  This observation motivates the introduction of the
following trigger sets, which cause neighbors to send synchronized
broadcasts periodically to an agent if its state remains at
zero. First, if an agent $i$'s state is zero and it has not received a
synchronized broadcast from its neighbors for $\tau_i$ time (here,
$\tau_i >0$ is a design parameter), it triggers a broadcast to notify
its neighbors that it requires new states. This behavior is captured
by the trigger set
\begin{align*}
  \dtriggersetrequest := 
  \begin{cases}
    \{ \xi \in \Xi: x_i = 0 \text{ and } s_i \ge \tau_i \}, \;
    &\text{{\rm if $i \le n$,}}
    \\
    \emptyset, \; &\text{{\rm if $i \ge n+1$.}}
  \end{cases}
\end{align*}
where we use the state $s_i$ to denote the time since $i$ has last
sent a broadcast. On the receiving end, if $i$ receives a broadcast
request from a neighbor~$j$, then it should also broadcast
immediately,
\begin{align*}
  \dtriggersetsend := \{\xi \in \Xi :\exists j \in \neigh_i^x \text{
    s.t. } q_{i,j} = 1 \}.
\end{align*}
where $q_{i,j} \in \{0,1\}$ is a state with $q_{i,j} = 1$ indicating
that $j$ has requested a broadcast from $i$.

Our last component of the distributed trigger design addresses the
problem posed by the asynchronism in state broadcasts. In fact, given
that agents determine autonomously when to communicate with their
neighbors, this may cause non-persistence in the resulting network
evolution. As an example, consider a scenario where successive state
broadcasts by one agent cause another neighboring agent to generate a
state broadcast of its own after increasingly smaller time intervals,
and vice versa.  To address this problem, we provide a final component
to the design of the distributed trigger set as follows,
\begin{align*}
  \dtriggersetsynch := \{ \xi \in \Xi: 0 \le r_i \le r^{\min}_i \},
\end{align*}
where $r_i$ represents the time elapsed between when agent $i$
received a state broadcast from a neighbor and $i$'s last
broadcast. We use $r_i = -1$ to indicate that $i$ has not received a
broadcast from a neighbor since its own last state broadcast. The
threshold $r^{\min}_i>0$ is a design parameter (smaller values result
in less frequent updates).  Intuitively, this trigger means that if an
agent broadcasts its state and in turn receives a state broadcast from
a neighbor faster than some tolerated rate, the agent broadcasts its
state immediately again. The effect of this trigger is that, if
broadcasts start occurring too frequently in the network, neighboring
agents' broadcasts synchronize. This emergent behavior is described in
more depth in the proof of Theorem~\ref{th:convergence_dist} later.

Finally, the overall distributed trigger set for each $i \in
\{1,\dots,n+m\}$ is,
\begin{align}
  \dtriggerset := \dtriggersete \cup \dtriggersetzero \cup
  \dtriggersetrequest \cup \dtriggersetsend \cup
  \dtriggersetsynch. \label{eq:trigger_dist}
\end{align}

\subsection{Distributed algorithm and convergence analysis}

We now state the distributed algorithm and its convergence properties
which are the main contributions of this paper. 

\begin{algorithmdefinition}\longthmtitle{Distributed linear
    programming with event-triggered communication}\label{algo}
  For each agent $i \in \{1,\dots,n + m \}$, if $\xi \notin
  \mathcal{T}_i$ then
  \begin{subequations}\label{eq:hybrid3}
    \begin{alignat}{2}
      \dot{x}_i &=
      \begin{cases}
        \nomflow_i(\hat{x},\hat{z}), \hspace{14.3mm} \text{{\rm if
            $\hat{x}_i > 0$,}}
        \\
        \max \{0,\nomflow_i(\hat{x},\hat{z})\}, \; \text{{\rm if
            $\hat{x}_i = 0$,}}
      \end{cases}
      \qquad &&\text{{\rm if $i \le n$}}
      \\
      \dot{z}_{i-n} &= a_{i-n}^T\hat{x}-b_{i-n}, \qquad &&\text{{\rm
          if $i \ge n+1$}}
      \\
      \dot{s}_i &=
      \begin{cases}
        1, \quad \text{{\rm if $s_i < \tau_i$,}}
        \\
        0, \quad \text{{\rm if $s_i \ge \tau_i$,}}
       \end{cases} \qquad &&\text{{\rm
          for all $i$}}
    \end{alignat}
    and, if $\xi \in \mathcal{T}_i$, then
    \begin{alignat}{2}
      \hat{x}_i^+ &= x_i \qquad &&\text{{\rm if $i \le n$}}
      \\
      \hat{z}_{i-n}^+ &= z_{i-n}, \qquad &&\text{{\rm if $i \ge n+1$}}
      \\
      (s_i^+,r_i^+,r_j^+) &= (0,-1,s_j), \qquad &&\text{{\rm for all
          $i$ and all $j \in \neigh_i$}}
      \\
      q_{j,i}^+ &= 1, \qquad &&\text{{\rm if $\xi \in
          \dtriggersetrequest$ and for all $j \in \neigh_i$}}
      \\
      q_{i,j}^+ &= 0, \qquad &&\text{{\rm if $\xi \in
          \dtriggersetsend$ and for all $j \in \neigh_i^x$}}
    \end{alignat}
  \end{subequations}
\end{algorithmdefinition}
The entire network state is given by $\xi \in \Xi$. However, the local
state of an individual agent $i \in \{1,\dots,n\}$ consists of
$x_i,\hat{x}_i,s_i,r_i,$ and $\cup_{j \in
  \neigh^x_i}\{q_{i,j}\}$. Likewise, the local state of agent $i \in
\{n+1,\dots,n+m\}$ consists of $z_{i-n},\hat{z}_{i-n},s_i,r_i,$ and
$\cup_{j \in \neigh^x_i}\{q_{i,j}\}$. These latter agents may be
implemented as virtual agents as described in
Remark~\ref{rem:distributed-implementation}. Then, recalling the
assumptions on local information outlined in
Section~\ref{sec:problem-statement}, it is straightforward to see that
the coordination algorithm~\eqref{eq:hybrid3} can be implemented by
the agents in a distributed way. We are now ready to state our main
convergence result.

\begin{theorem}\longthmtitle{Distributed triggers - convergence and
    persistently flowing solutions}\label{th:convergence_dist}
  For each $i \in \{1,\dots,n+m\}$, let $0 < \mu_i \le \frac{1}{160}$
  and
  \begin{align*}
    0 < r^{\min}_i \le \tau_i < \frac{1}{\sqrt{960 |\neigh_i|
        \max_{j \in \neigh_i}|\neigh_j|}}.
  \end{align*}
  Let $\psi$ be a solution of~\eqref{eq:hybrid3}, with each set
  $\dtriggerset$ defined by~\eqref{eq:trigger_dist}. Then,
  \begin{enumerate}
  \item if $\psi$ is persistently flowing, there exists a point
    $(x_*,z) \in \primalsol \times \reals^m$ such that,
    \begin{gather*}
      (x(t,j),z(t,j)) \rightarrow (x_*,z') \quad \text{as} \quad t+j
      \longrightarrow \infty, \quad (t,j) \in \dom(\psi) ,
    \end{gather*}
  \item if there exists $\delta_P > 0$ such that, for any time
    $(t',j') \in \dom(\psi)$ where $\psi(t',j') \in \dtriggersetzero$
    for some $i \in \until{n}$, it holds that $\psi(t,j) \notin
    \dtriggersetzero$ for all $(t,j) \in ((t',t'+\delta_P] \times
    \naturals) \cap \dom(\psi)$, the solution $\psi$ is persistently
    flowing.
  \end{enumerate}
\end{theorem}
\begin{proof}
  The proof of the convergence result in (i) follows closely the
  argument we employed to establish
  Theorem~\ref{th:convergence_prelim}. One key difference is that the
  intervals on which $\psi$ flows do not necessarily correspond to the
  intervals on which $\mathcal{L}_FV$ exists. This is because the
  value of $\sigma$ may change even though $\psi$ still
  flows. However, since the dynamics $(\dot{x},\dot{z}) =
  F(\hat{x},\hat{z})$ is constant on periods of flow it is easy to see
  that there can be at most $n$ agents added to $\sigma$ in any given
  period of flow. This means that, if $\psi$ is persistently flowing
  according to the characterization (PFii), the Lie derivative
  $\mathcal{L}_FV$ exists persistently often for periods of length
  $\tau_P / n$ (since $\sigma$ must be constant for an interval of
  length at least $\tau_P / n$ persistently often). Thus, let us
  consider a time $(t,j)$ such that $(t,j) \in (t_j,t_j+\tau_p / n )
  \times \{j\} \subset \dom(\psi)$ and $\mathcal{L}_FV(x(t,j),z(t,j))$
  exists. Note that, if $\psi$ is persistently flowing according to
  the characterization (PFi), we may take $\tau_P = \infty$ and the
  following analysis holds. To ease notation, denote $p(t,j) =
  \sigma(x(t,j),z(t,j))$ and $\hat{p}(t,j) =
  \sigma(\hat{x}(t,j),\hat{z}(t,j))$. Then, following the exposition
  in the proof of Theorem~\ref{th:convergence_prelim}, one can see
  that, due to trigger sets $\dtriggersete$ and $\dtriggersetzero$ and
  the conditions on $\mu_i$,
  \begin{align}\label{eq:convergence2_bound}
    \mathcal{L}_FV(x(t,j),z(t,j)) &\le
    -\frac{1}{4}\nomflow(\hat{x}(t,j),\hat{z}(t,j))^TI_{\hat{p}(t,j)}
    \nomflow( \hat{x}(t,j),\hat{z}(t,j)) \nonumber
    \\
    & \quad - \frac{1}{8}(A\hat{x}(t,j)-b)^T(A\hat{x}(t,j)-b)
    \nonumber
    \\
    &\quad + 15 \nomflow(x(t,j),z(t,j)) I_{p(t,j) \setminus
      \hat{p}(t,j)} \nomflow(x(t,j),z(t,j)) .
  \end{align}
  We focus on the last term, which is the only positive one. If $i \in
  p(t,j) \setminus \hat{p}(t,j)$, then it must be that $\hat{x}_i = 0$
  and thus $i$ is receiving state broadcasts from its neighbors every
  $\tau_i$ seconds by design of $\dtriggersetrequest$ and
  $\{\mathcal{T}^{\operatorname{send}}_j\}_{j \in
    \neigh_i}$. Therefore, the maximum amount of time that any $i$
  remains in $p(t,j) \setminus \hat{p}(t,j)$ is $\tau_i$
  seconds. Since each $\tau_i < \tfrac{1}{\sqrt{960 |\neigh_i|
      \max_{j \in \neigh_i}|\neigh_j|}}$, we apply
  Proposition~\ref{prop:mismatch} using $\nu < \tfrac{2
    \sqrt{2}}{\sqrt{960 |\neigh_i| \max_{j \in \neigh_i}|\neigh_j|}}$
  to obtain
  \begin{multline*}
    \nomflow_i(x(t,j),z(t,j))^2 < \frac{1}{120 |\neigh_i| \max_{j
        \in \neigh_i}|\neigh_j|}
    \big(\nomflow(\hat{x}(t,j),\hat{z}(t,j))^TI_{\hat{p}(t,j) \cap
      \neigh^x_i} \nomflow(\hat{x}(t,j),\hat{z}(t,j))
    \\
    + (A\hat{x}(t,j)-b)^TI_{\neigh^z_i}(A\hat{x}(t,j)-b) \big) .
  \end{multline*}
  From the above bound it is clear that
  \begin{align*}
    &\nomflow(x(t,j),z(t,j)) I_{p(t,j) \setminus \hat{p}(t,j)}
    \nomflow(x(t,j),z(t,j))
    \\
    &< \frac{1}{120} \sum_{i \in p(t,j) \setminus \hat{p}(t,j)}
    \frac{1}{|\neigh_i| \max_{j \in \neigh_i}|\neigh_j|}
    \big(\nomflow(\hat{x}(t,j),\hat{z}(t,j))^TI_{\hat{p}(t,j) \cap
      \neigh^x_i} \nomflow(\hat{x}(t,j),\hat{z}(t,j))
    \\
    & \hspace{60mm} +
    (A\hat{x}(t,j)-b)^TI_{\neigh^z_i}(A\hat{x}(t,j)-b) \big)
    \\
    &< \frac{1}{120} \sum_{i \in \{1,\dots,n\}} \frac{1}{|\neigh_i|}
    \sum_{k \in \neigh_j} \frac{1}{|\neigh_j|}
    \big(\nomflow(\hat{x}(t,j),\hat{z}(t,j))^TI_{\hat{p}(t,j)}
    \nomflow(\hat{x}(t,j),\hat{z}(t,j))
    \\
    & \hspace{50mm} + (A\hat{x}(t,j)-b)^T(A\hat{x}(t,j)-b) \big)
    \\
    &< \frac{1}{120}\big( \nomflow(\hat{x}(t,j),\hat{z}(t,j))^TI_{
      \hat{p}(t,j)} \nomflow(\hat{x}(t,j),\hat{z}(t,j)) +
    (A\hat{x}(t,j)-b)^T(A\hat{x}(t,j)-b) \big) ,
  \end{align*}
  which, when combined with~\eqref{eq:convergence2_bound}, reveals
  that there exists some $\epsilon > 0$ such that
  \begin{align*}
    &\mathcal{L}_FV(x(t,j),z(t,j))
    \\
    &\le -\epsilon \big( \nomflow(\hat{x}(t,j),\hat{z}(t,j))^TI_{
      \hat{p}(t,j)} \nomflow(\hat{x}(t,j),\hat{z}(t,j)) +
    (A\hat{x}(t,j)-b)^T(A\hat{x}(t,j)-b) \big) .
  \end{align*}
  The remainder of the convergence proof now follows in the same way
  as the proof of Theorem~\ref{th:convergence_prelim}.
  
  Next, we prove (ii) by contradiction. Suppose that the conditions in
  (ii) are satisfied but $\psi$ is not persistently flowing. Then, for
  any $\epsilon > 0$, there exists $T_{\epsilon}$ such that for every
  $(t,j) \in \dom(\psi)$ with $t+j\ge T_{\epsilon}$, the time between
  state broadcasts is less than $\epsilon$. Choose
  \begin{align*}
    \epsilon < \min\Big\{\frac{1}{n+1}\min_{i}
    r^{\min}_i,\min_i\tau_i,\frac{1}{n+1}\delta_P,\min_i \sqrt{\mu_i}
    \Big\}.
  \end{align*}
  Then, we can show that all the state broadcasts in the network are
  synchronized from $T_{\epsilon}$ time forward due to the trigger
  sets $\dtriggersetsynch$: by our choice of
  $\epsilon$, there are at least $(n+1)\epsilon$ broadcasts every
  $\min_ir^{\min}_i$ seconds. This means that at least one agent has
  broadcast twice in the last $\min_ir^{\min}_i$ seconds. Accordingly,
  all the neighbors of that agent synchronously broadcast their state
  at the same time due to the trigger
  sets~$\dtriggersetsynch$. Propagating this logic
  to the second-hop neighbors and so on, one can see that the entire
  network is synchronously broadcasting its state and this will be
  true for all $t+j \ge T_{\epsilon}$. Let us then explore the
  possible causes of the next broadcast. Clearly, the next broadcast
  will not be due to any $\dtriggersetsynch$, since
  agent broadcasts are synchronized already. Likewise, it will not be
  due to any $\dtriggersetrequest$ or
  $\dtriggersetsend$ since, by construction,
  $\min_i\tau_i > \epsilon$ time must have elapsed before
  $\dtriggersetrequest$ is enabled by any agent. By
  assumption, only $n$ broadcasts due to the $\dtriggersetzero$ can
  occur in $\delta_P$ seconds.  Without loss of generality, we can
  assume that the next broadcast due to one of $\dtriggersetzero$ does
  not occur for another $\frac{\delta_P}{n+1} > \epsilon$ time. This
  leaves the $\dtriggersete$ trigger sets. Let us look at the
  evolution of $(e_x)_i$ for any given $i \in \{1,\dots,n\}$. Since
  $i$ has not received a broadcast from its neighbors, the evolution
  of $(e_x)_i$ is
  \begin{align*}
    (e_x)_i(t,j) = \nomflow_i(\hat{x},\hat{z})(t-t_j) .
  \end{align*}
  Therefore, for $\dtriggersete$ to have been enabled, $\min_i
  \sqrt{\mu_i} > \epsilon$ time must have elapsed (the same conclusion
  holds for $i \in \{n+1,\dots,n+m\}$). This means that the next
  broadcast is not triggered in $\epsilon$ time, contradicting the
  definition of $\epsilon$, and this completes the proof.
\end{proof}

As shown in the proof of Theorem~\ref{th:convergence_dist}, the
triggers defined by $\dtriggersete$, $\dtriggersetrequest$,
$\dtriggersetsend$, and $\dtriggersetsynch$ do not cause
non-persistency in the solutions of~\eqref{eq:hybrid3}. If we had
used~\eqref{eq:linprog} in our derivation instead of~\eqref{eq:pert},
the resulting design would not have enjoyed this attribute,
cf. Remark~\ref{re:motivation-qp}. In our experience, the hypothesis
in Theorem~\ref{th:convergence_dist}(ii) is always satisfied with
$\delta_P = \infty$, which suggests that all solutions
of~\eqref{eq:hybrid3} are persistently flowing.

\begin{remark}\longthmtitle{Robustness to
    perturbations} \label{rem:robustness} {\rm We briefly comment here
    on the robustness properties of the coordination
    algorithm~\eqref{eq:hybrid3} against perturbations.  These may
    arise in the execution as a result of communication noise,
    measurement error, modeling uncertainties, or disturbances in the
    dynamics, among other reasons.  A key advantage of modeling the
    execution of the coordination algorithm in the hybrid systems
    framework described Section~\ref{sec:hybrid-systems} is that there
    exist a suite of established robustness characterizations for such
    systems. In particular, it is fairly straightforward to verify
    that~\eqref{eq:hybrid3} is a `well-posed' hybrid system, as
    defined in~\cite{RG-RGS-ART:09}, and as a consequence of this
    fact, the convergence properties stated in
    Theorem~\ref{th:convergence_dist}(i) remain valid if the hybrid
    system~\eqref{eq:hybrid3} is subjected to sufficiently small
    perturbations (see e.g.,~\cite[Theorem 7.21]{RG-RGS-ART:09}).
    Moreover, in our previous work~\cite{DR-JC:13-tac}, we have shown
    that the continuous-time dynamics~\eqref{eq:disc-dyn} (upon which
    our distributed algorithm with event-triggered communication is
    built) is integral-input-to-state stable, and thus robust to
    disturbances of finite energy.  We believe that the coordination
    algorithm~\eqref{eq:hybrid3} inherits this desirable property,
    although we do not characterize this explicitly here for reasons
    of space.  Nevertheless, Section~\ref{sec:sims} below illustrates
    the algorithm performance under perturbation in simulation.}
  \oprocend
\end{remark}

\section{Simulations}\label{sec:sims}

Here we illustrate the execution of the coordination algorithm
\eqref{eq:hybrid3} with event-triggered communication in a multi-agent
assignment example. The multi-agent assignment problem we consider is
a resource allocation problem where $N$ tasks are to be assigned to
$N$ agents. Each potential assignment of a task to an agent has an
associated benefit and the global objective of the network is to
maximize the sum of the benefits in the final assignment. The
assignment of an agent to a task is managed by a broker and the set of
all brokers use the strategy~\eqref{eq:hybrid3} to find the optimal
assignment. Presumably, a broker is only concerned with the
assignments of the agent and task that it manages and not the
assignments of the entire network. Additionally, there may exist
privacy concerns that limit the amount of information that a network
makes available to any individual broker. These are a couple of
reasons why a distributed algorithm is well-suited to solve this
problem.

\begin{figure*}[hbt!]
  \centering
  \subfigure[Assignment graph]{\includegraphics[trim=0 -35 0 0,width=0.3\linewidth]{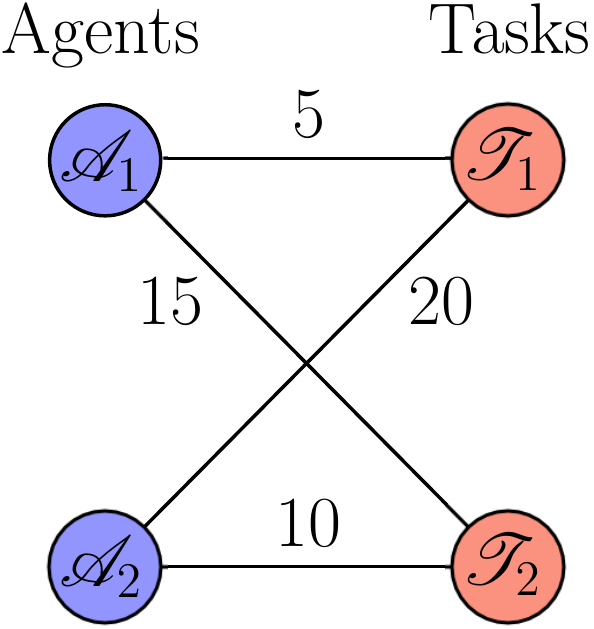}} \hspace{20mm}
  \subfigure[Communication between brokers]{\includegraphics[width=0.4\linewidth]{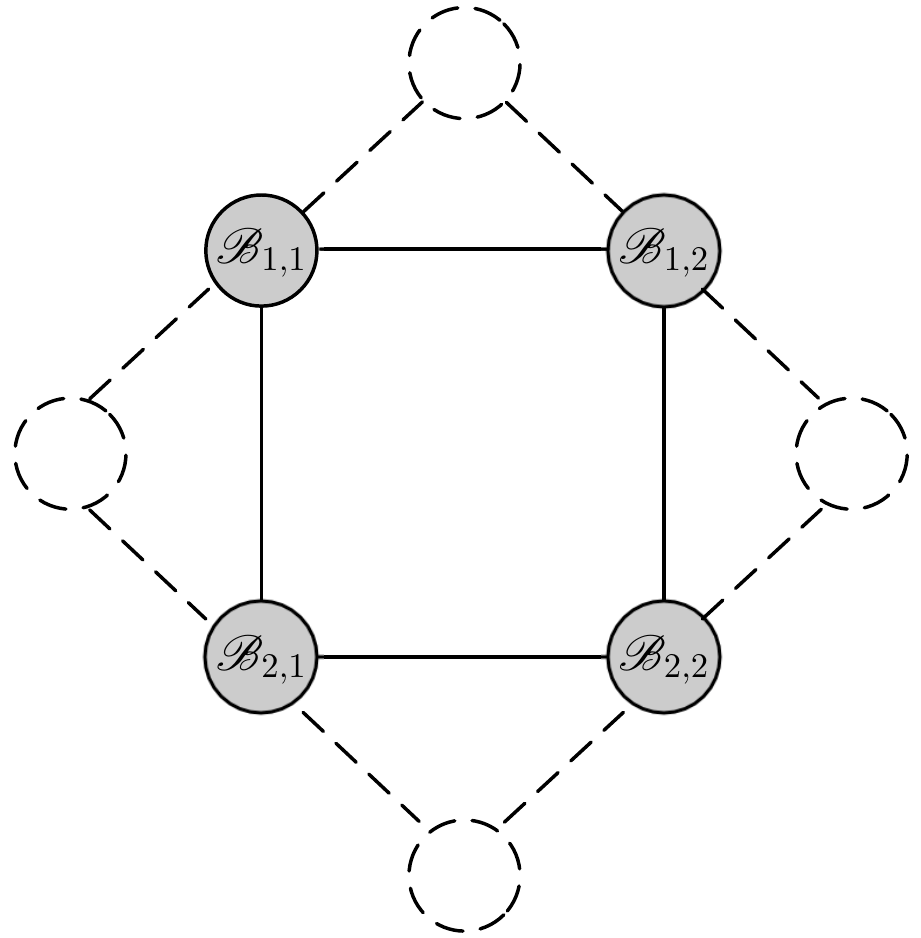}}
  \caption{(a) shows the assignment graph with agents $\mathscr{A}_1$
    and $\mathscr{A}_2$ in blue, tasks $\mathscr{T}_1$ and
    $\mathscr{T}_2$ in red, and the benefit of a potential assignment
    as edge weights.  (b) shows the connectivity among brokers. Broker
    $\mathscr{B}_{i,j}$ is responsible for determining the potential
    assignment of task $\mathscr{T}_j$ to agent $\mathscr{A}_i$.  The
    dashed nodes represent the virtual brokers whose states correspond
    to the components of the Lagrange multipliers~$z$
    in~\eqref{eq:hybrid3}, see
    Remark~\ref{rem:distributed-implementation}.}\label{fig:sim_graphs}
\end{figure*}

We consider an assignment problem with $2$ agents (denoted by
$\mathscr{A}_1$ and $\mathscr{A}_2$) and $2$ tasks (denoted by
$\mathscr{T}_1$ and $\mathscr{T}_2$) as shown in
Figure~\ref{fig:sim_graphs}(a). The assignment problem is to be solved
by a set of $4$ brokers as shown in Figure~\ref{fig:sim_graphs}(b). In
general, the number of brokers is the number of edges in the
assignment graph. Broker $\mathscr{B}_{i,j}$ is responsible for
determining the potential assignment of task $\mathscr{T}_j$ to agent
$\mathscr{A}_i$ and has state $x_{i,j} \in \{0,1\}$. Here, $x_{i,j} =
1$ means that task $\mathscr{T}_j$ is assigned to agent
$\mathscr{A}_i$ (with associated benefit $c_{i,j} \in
\realsnonnegative$) and $x_{i,j} = 0$ means that they are not assigned
to each other.  We formulate the multi-agent assignment problem as the
following optimization problem,
\begin{subequations}\label{eq:assignment_prob-integer}
  \begin{alignat}{2}
    &\max && \quad c_{1,1} x_{1,1} + c_{1,2} x_{1,2} + c_{2,1} x_{2,1}
    + c_{2,2} x_{2,2}
    \\
    &\hspace{1.5mm}\text{s.t.} && \quad x_{1,1} + x_{1,2} =
    1 \label{eq:sim_con1}
    \\
    & && \quad x_{2,1} + x_{2,2} = 1 \label{eq:sim_con2}
    \\
    & && \quad x_{1,1} + x_{2,1} = 1 \label{eq:sim_con3}
    \\
    & && \quad x_{1,2} + x_{2,2} = 1 \label{eq:sim_con4}
    \\
    & && \quad x_{1,1},x_{1,2},x_{2,1},x_{2,2} \in \{0,1\}
    . \label{eq:sim_con5}
  \end{alignat}
\end{subequations}
Constraints~\eqref{eq:sim_con1}-\eqref{eq:sim_con2}
(resp.~\eqref{eq:sim_con3}-\eqref{eq:sim_con4}) ensure that each agent
(resp. task) is assigned to one and only one task (resp. agent). Note
that the connectivity between brokers shown in
Figure~\ref{fig:sim_graphs}(b) is consistent with the requirements for
a distributed implementation as specified by the constraint equations
of~\eqref{eq:assignment_prob-integer}.  It is known, see
e.g.,~\cite{AS:00}, that the relaxation $x_{i,j} \ge 0$ of the
constraints~\eqref{eq:sim_con5} gives rise to a linear program with an
optimal solution that satisfies $x_{i,j} \in \{0,1\}$. Thus, for our
purposes, we solve instead the following linear program
\begin{subequations}\label{eq:assignment_prob}
  \begin{alignat}{2}
    &\min && \quad -5 x_{1,1} - 15 x_{1,2} - 20 x_{2,1} - 10 x_{2,2}
    \\
    &\hspace{1.5mm}\text{s.t.} && \quad x_{1,1} + x_{1,2} = 1
    \\
    & && \quad x_{2,1} + x_{2,2} = 1
    \\
    & && \quad x_{1,1} + x_{2,1} = 1
    \\
    & && \quad x_{1,2} + x_{2,2} = 1
    \\
    & && \quad x_{1,1},x_{1,2},x_{2,1},x_{2,2} \ge 0 ,
  \end{alignat}
\end{subequations}
where we have also converted the maximization into a minimization by
considering the negative of the objective function and substituted the
values of the benefits given in
Figure~\ref{fig:sim_graphs}(a). Clearly, the linear
program~\eqref{eq:assignment_prob} is in standard form. Its solution
set is $\primalsol = \{x^*\}$, with $x^*=
(x_{1,1}^*,x_{1,2}^*,x_{2,1}^*,x_{2,2}^*) = (0,1,1,0)$, corresponding
to the optimal assignment consisting of the pairings
$(\mathscr{A}_1,\mathscr{T}_2)$ and $(\mathscr{A}_2,\mathscr{T}_1)$.

\begin{figure*}[hbt!]
  \centering
  \subfigure[Broker state trajectories]{\includegraphics[width=0.49\linewidth]{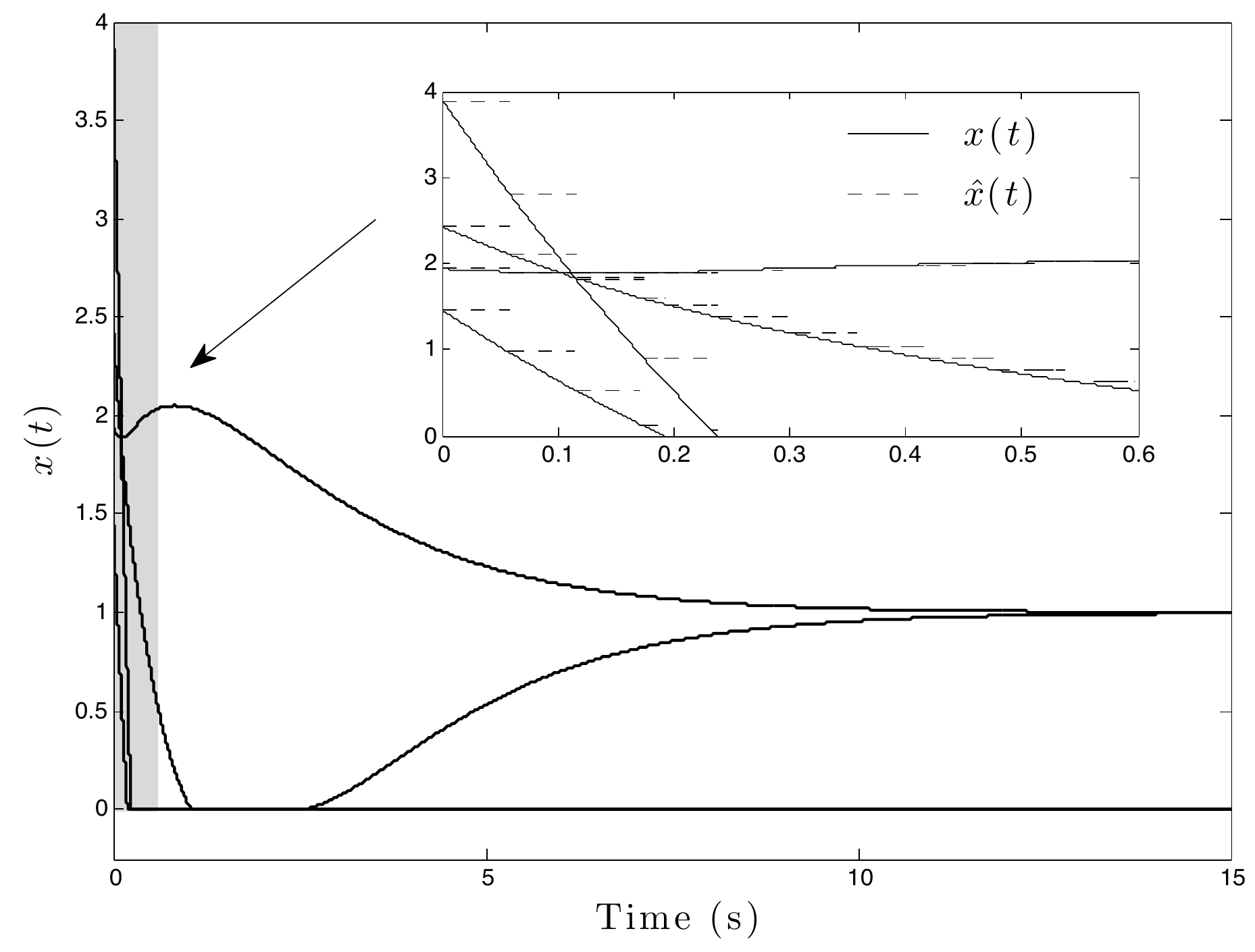}}
  \subfigure[Virtual broker state trajectories]{\includegraphics[width=0.49\linewidth]{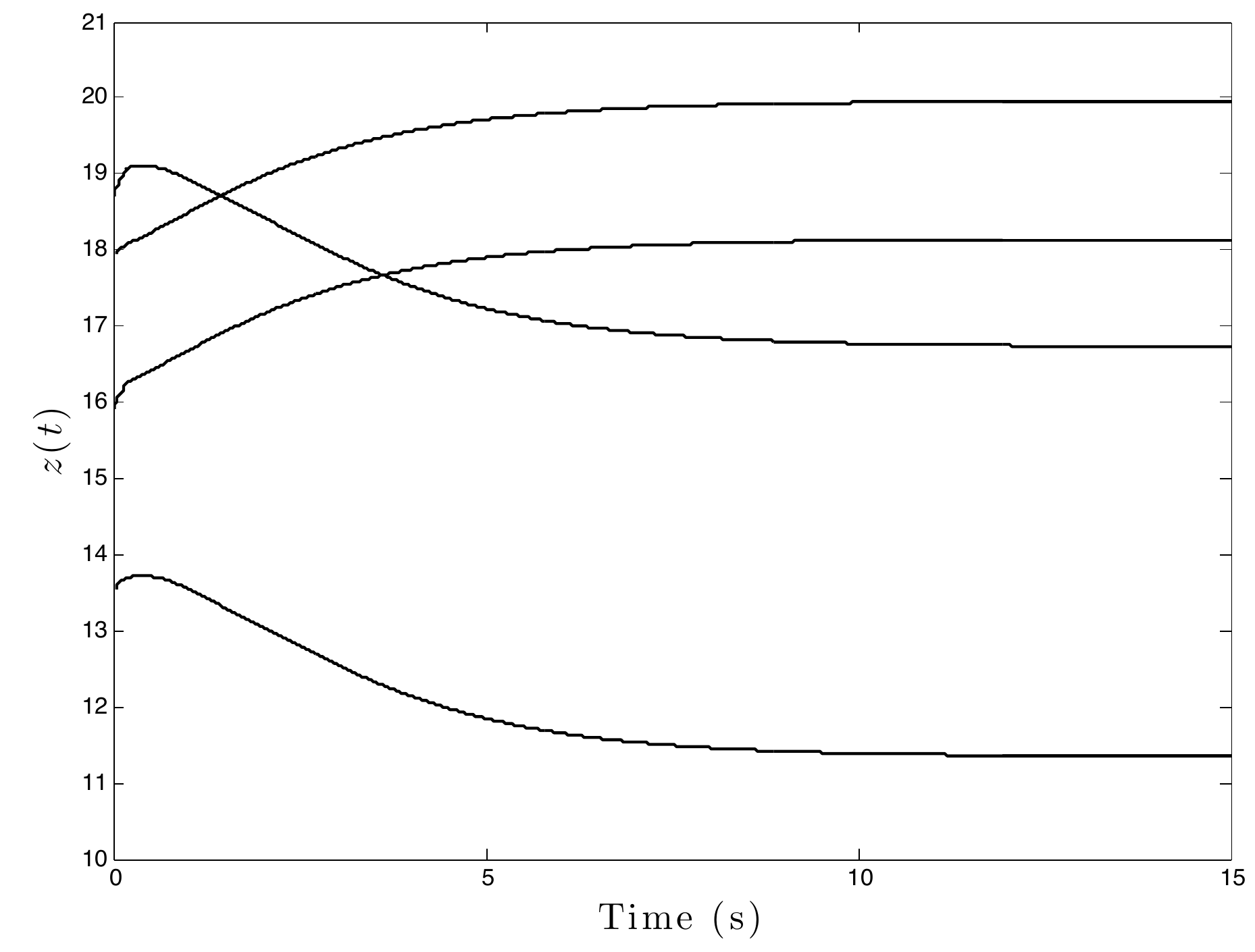}}
  \subfigure[Evolution of $V$]{\includegraphics[width=0.49\linewidth]{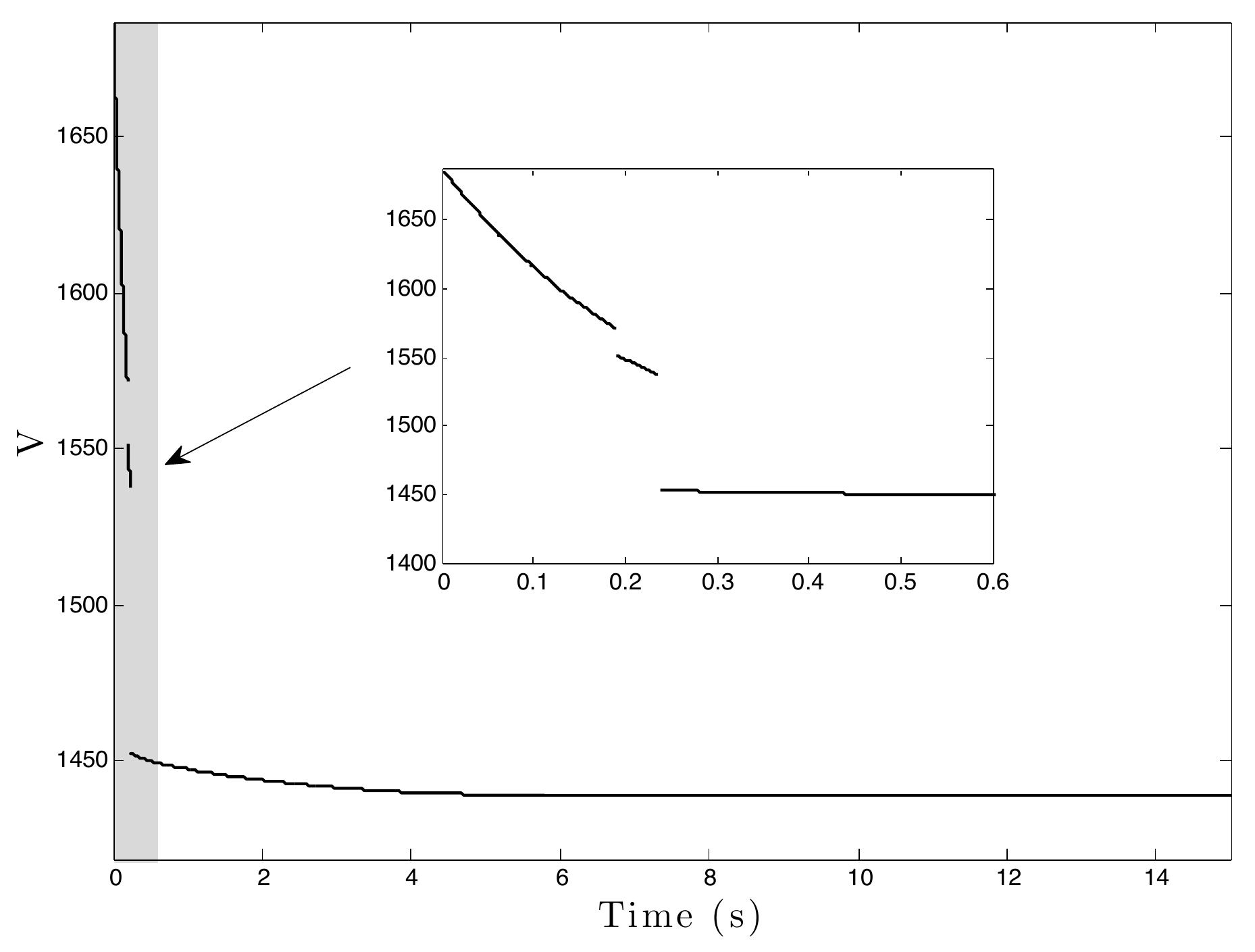}}
  \subfigure[Cumulative number of broadcasts]{\includegraphics[width=0.49\linewidth]{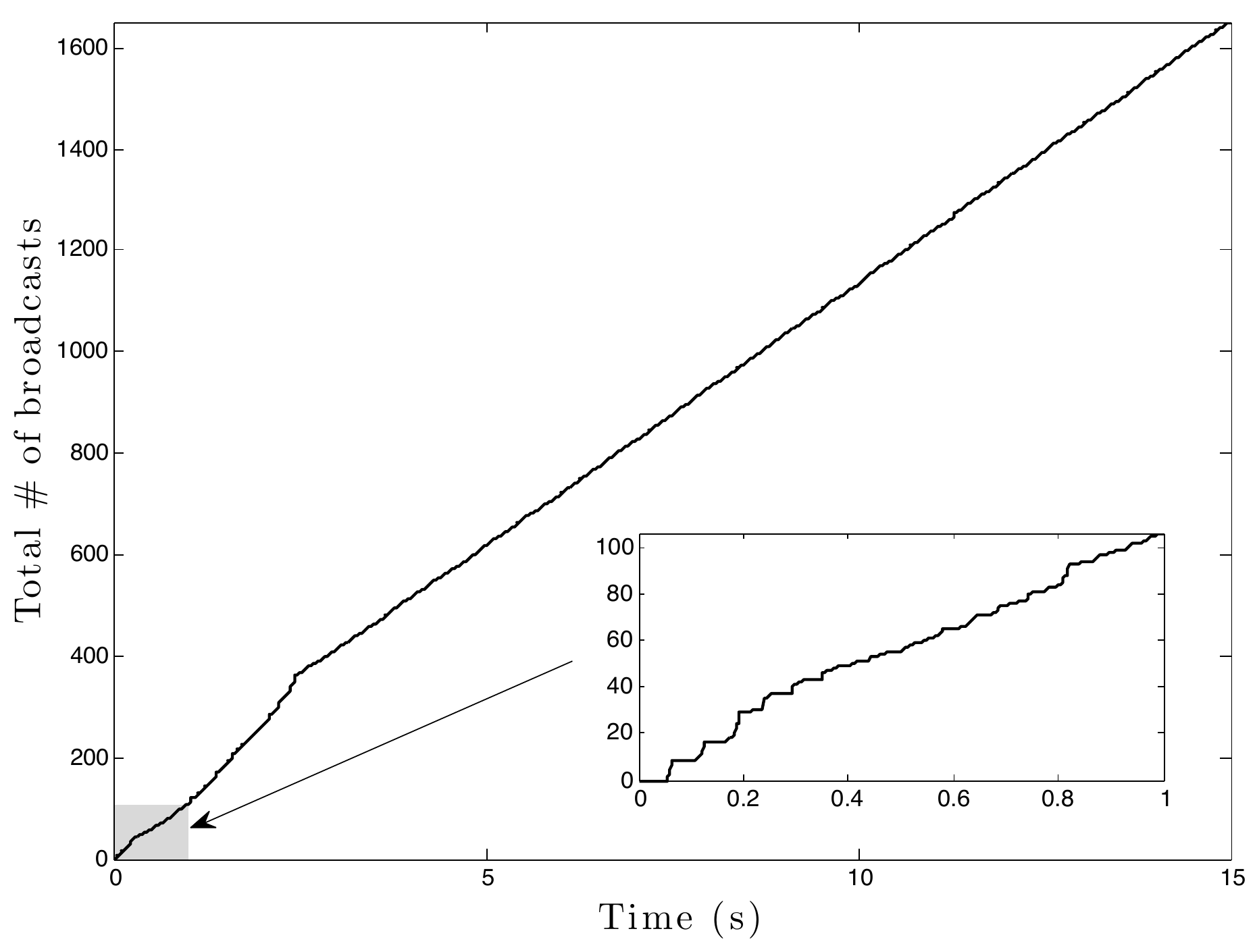}}
  \caption{Simulation results of brokers
    implementing~\eqref{eq:hybrid3} to solve the multi-agent
    assignment problem~\eqref{eq:assignment_prob}. (a) shows the state
    trajectories of the brokers, with an inlay displaying the
    transient response in detail. The brokers' state is $x =
    (x_{1,1},x_{1,2},x_{2,1},x_{2,2})$ and the inlay also shows the
    evolution of the broadcast states, $\hat{x} =
    (\hat{x}_{1,1},\hat{x}_{1,2},\hat{x}_{2,1},\hat{x}_{2,2})$ in
    dashed lines. The aggregate of the brokers' states converge to the
    unique solution $\primalsol = \{(0,1,1,0)\}$. (b) shows the
    evolution of the virtual brokers' states. The Lyapunov function
    $V$ is discontinuous but decreasing, as evidenced in (c). The
    cumulative number of broadcasts appears roughly linear and the
    execution is clearly persistently flowing. Each inlay shows the
    transient in detail.} \label{fig:sims}
\end{figure*}

Figure~\ref{fig:sims} shows the group of brokers executing the
distributed coordination algorithm~\eqref{eq:hybrid3} with
event-triggered communication.  Figure~\ref{fig:noisy_sims}
illustrates the algorithm performance in the presence of additive
white noise on the state broadcasts. The convergence in this case
shows the algorithm robustness to sufficiently small disturbances, as
pointed out in Remark~\ref{rem:robustness}.

\begin{figure*}[hbt!]
  \centering
 \subfigure[Noisy broker state trajectories]{\includegraphics[width=0.49\linewidth]{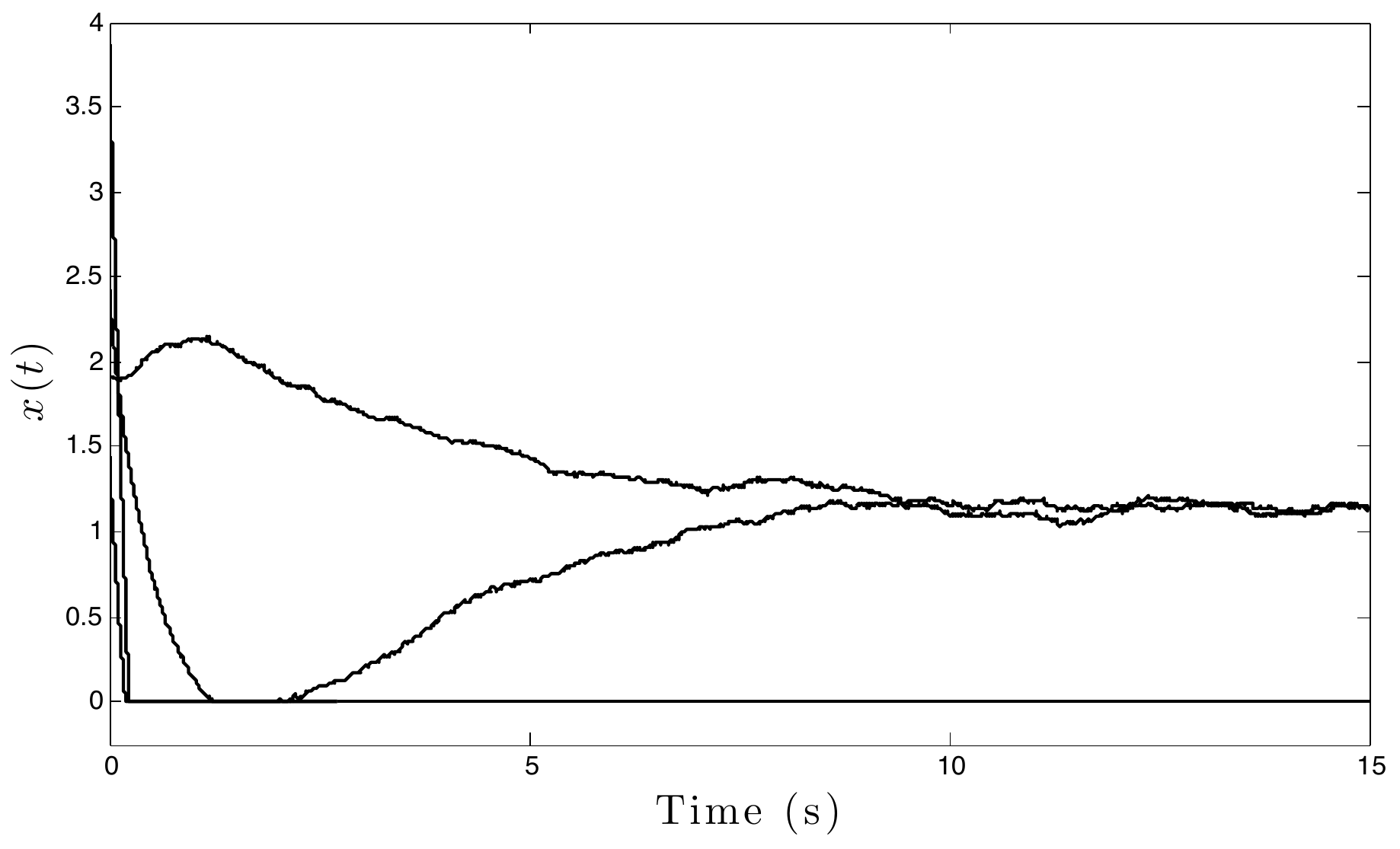}}
 \subfigure[Noisy virtual broker state trajectories]{\includegraphics[width=0.49\linewidth]{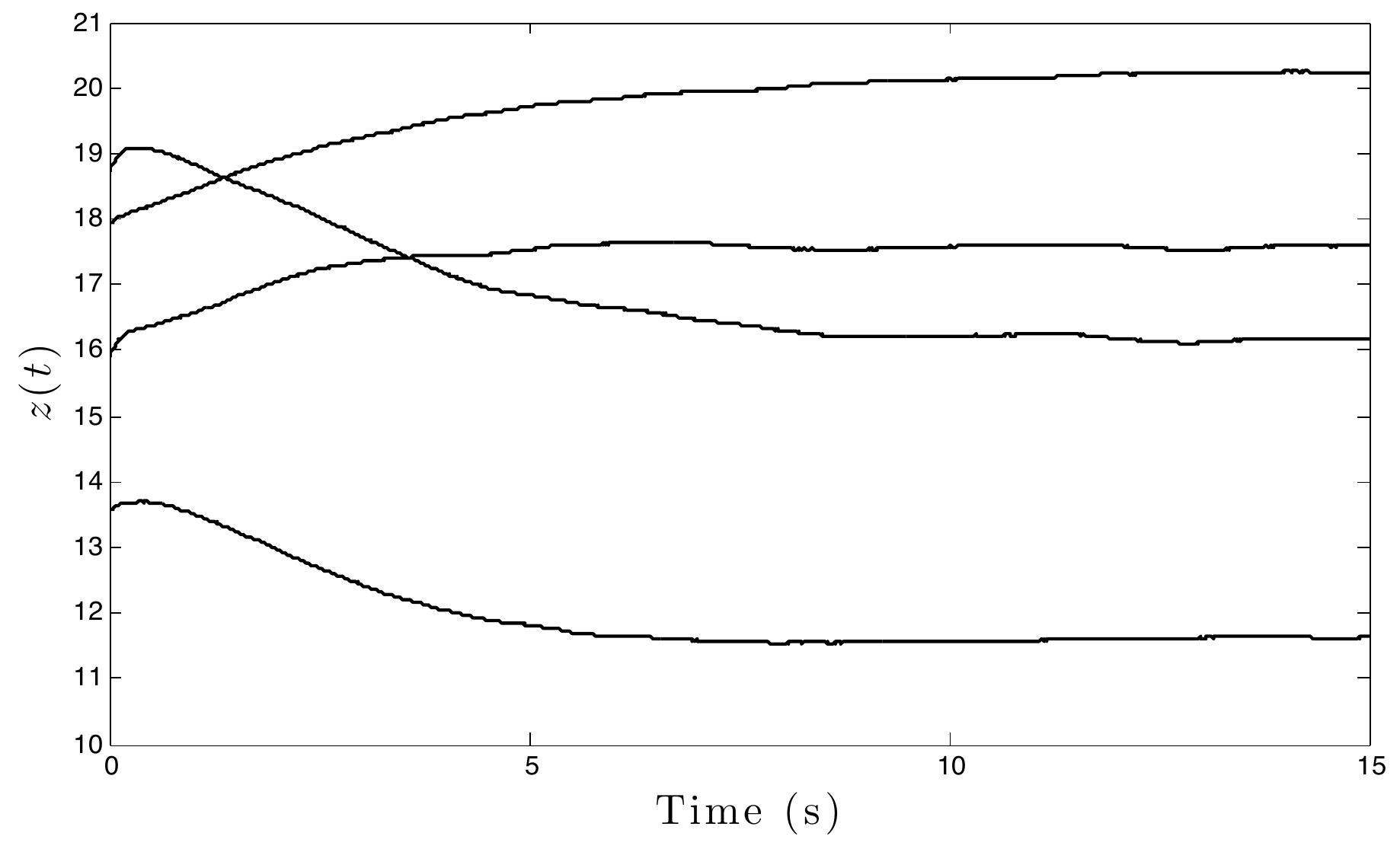}}
 \caption{Simulation results of brokers
   implementing~\eqref{eq:hybrid3} to solve the multi-agent assignment
   problem~\eqref{eq:assignment_prob} under additive noise in the
   communication channels.  In this simulation, broadcasts of
   information are corrupted by noise which is normally distributed
   with zero mean and standard deviation~$1$.  The network state
   converges to a neighborhood of the optimal solution and the optimal
   assignment can easily be deduced.} \label{fig:noisy_sims}
\end{figure*}

\myclearpage

\section{Conclusions and future work}\label{sec:conclusions}

We have studied the design of distributed algorithms for networks of
agents that seek to collectively solve linear programs in standard
form and rely on discrete-time communication. Our algorithmic solution
has agents executing a distributed continuous-time dynamics and
deciding in an opportunistic and autonomous way when to broadcast
updated state information to their neighbors. Our methodology combines
elements from linear programming, switched and hybrid systems,
event-triggered control, and Lyapunov stability theory to provide
provably correct centralized and distributed strategies.  We have
rigorously characterized the asymptotic convergence of persistently
flowing executions to a solution of the linear program. We have also
identified a sufficient condition for executions to be persistently
flowing, and based on it, we conjecture that they all are.  Future
work will be devoted to establish that all solutions are persistently
flowing, rigorously characterize the input-to-state stability
properties of the proposed algorithm, extend our approach to
event-triggered strategies for general switched systems, and implement
the results on a multi-agent testbed.

\Appendix\section*{}

The following two results are used extensively in the proof of
Proposition~\ref{prop:Lie-V} and elsewhere in the paper.

\begin{lemma}\longthmtitle{Young's
    inequality~\cite{GHH-JEL-GP:52}}\label{lem:bilinear}  
  Let $d_1,d_2 \in \naturals$ and $\mu \in \reals^{d_1}$, $M \in
  \reals^{d_1 \times d_2}$, $\nu \in \reals^{d_2}$. Then, for any
  $\kappa > 0$,
  \begin{align*}
    \mu^T M \nu &\le \frac{\kappa}{2} \nu^T M^T M \nu +
    \frac{1}{2\kappa} \mu^T\mu.
  \end{align*}
\end{lemma}

\begin{theorem}\longthmtitle{Cauchy Interlacing Theorem~\cite[Theorem
    4.3.15]{RAH-CRJ:85}}\label{th:interlacing}
  For a matrix $0 \preceq A \in \reals^{d \times d}$, let $0 \le
  \lambda_1 \le \dots \le \lambda_d$ denote its eigenvalues.  For $p
  \in \{1,\dots,d\}$, let $A_{p}$ be the matrix obtained by zeroing
  out the $p^{\text{th}}$ row and column of~$A$, and let $0 = \mu_1
  \le \dots \le \mu_d$ denote its eigenvalues. Then $\mu_1 \le
  \lambda_1 \le \mu_2 \le \lambda_{2} \le \dots \le \mu_{d} \le
  \lambda_{d}$.
\end{theorem}


\bibliographystyle{siam}%
\bibliography{alias,Main,Main-add,JC}

\end{document}